\documentclass[11pt,a4paper]{article}

\usepackage{times}
\usepackage{enumerate}

\usepackage[english]{babel}

\usepackage{verbatim}

\usepackage{color}

\usepackage{a4wide}

\usepackage{amsfonts}
\usepackage{hyperref}

\usepackage{amsmath}

\usepackage{amscd}
\usepackage{tabularx}
\usepackage{arydshln}
\usepackage{verbatim}
\usepackage{color}
\usepackage[all]{xy}
\usepackage{mathtools} 
\usepackage{ textcomp }

\newtheorem{theorem}{Theorem}[section]
\newtheorem{lemma}[theorem]{Lemma}
\newtheorem{proposition}[theorem]{Proposition}
\newtheorem{corollary}[theorem]{Corollary}
\newtheorem{Counter-example}[theorem]{Counter-example}
\newtheorem{remark}[theorem]{Remark}
\newtheorem{example}[theorem]{Example}
\newenvironment{proof}{\noindent{\it Proof.}}{\hfill$\blacksquare$}

\newtheorem{definition}[theorem]{Definition}
\usepackage{amsmath}

\usepackage{amssymb}

\begin{document}

\def\Q{\mathbb Q}
\def\R{\mathbb R}
\def\N{\mathbb N}
\def\Z{\mathbb Z}
\def\C{\mathbb C}
\def\S{\mathbb S}
\def\L{\mathbb L}
\def\H{\mathbb H}
\def\K{\mathbb K}
\def\X{\mathbb X}
\def\Y{\mathbb Y}
\def\Z{\mathbb Z}
\def\E{\mathbb E}
\def\J{\mathbb J}
\def\I{\mathbb I}
\def\T{\mathbb T}
\def\H{\mathbb H}

\title{Spacelike immersions in certain Lorentzian manifolds  \\ with  lightlike foliations}

\author{
Rodrigo Mor\'on\footnote{Corresponding author. \\
Both authors are  partially supported  by Spanish MICINN project PID2020-118452GB-100.
},\,\, Francisco J. Palomo}

\date{}

\maketitle

\begin{abstract}
\noindent The generalized Schwarzschild spacetimes are introduced as warped manifolds where the base is an open subset of $\R^2$ equipped with a Lorentzian metric and the fiber is a Riemannian manifold. This family includes physically relevant spacetimes closely related to models of black holes. The generalized Schwarzschild spacetimes are endowed with involutive distributions which provide foliations by lightlike hypersurfaces. In this paper, we study spacelike submanifolds immersed in the generalized Schwarzschild spacetimes, mainly, under the assumption that such submanifolds lie in a leaf of the above foliations. In this scenario, we provide an explicit formula for the mean curvature vector field and establish relationships between the extrinsic and intrinsic geometry of the submanifolds. We have derived several characterizations of the slices, and we delve into the specific case where the warping function is the radial coordinate in detail. This subfamily includes the Schwarzschild and Reissner-Nordstr\"om spacetimes.

\end{abstract}

\hyphenation{Lo-rent-zi-an}

\section{Introduction}

The $(m+2)$-dimensional exterior Schwarzschild spacetime with mass $\mathbf{M}\geq 0$ is equipped with the Lorentzian metric
$$
\widetilde{g}=-\Big(1- \frac{2\mathbf{M}}{r^{m-1}} \Big)dt^{2}+ \frac{1}{\Big(1- \frac{2\mathbf{M}}{r^{m-1}} \Big)}dr^{2}+ r^{2}g_{\S^{m}},
$$
where $(t,r)\in \R\times \R_{+}$  with $
r^{m-1}> 2\mathbf{M}
$
and $g_{\S^{m}}$ denotes the usual round metric of constant sectional curvature $1$ on the $m$-dimensional sphere $\S^{m}.$
This spacetime satisfies the vacuum Einstein equation and is spherically symmetric. Even more, these properties characterizes the Schwarzschild  spacetime. For  $\mathbf{M}= 0$, the Schwarzschild   metric reduces to the Minkowski metric in spherical terms, Example \ref{280723A}.  
The exterior Schwarzschild  spacetime  has several remarkable properties which we are interested in. 
\begin{enumerate}

\item For every lightlike vector field $\xi$ in the $tr$-half plane, the Levi-Civita connection $\widetilde{\nabla}$ of  the Schwarzschild  metric satisfies
$\widetilde{\nabla}\xi= \alpha \otimes \xi$ for some $1$-form $\alpha.$

\item The $(m+2)$-dimensional exterior Schwarzschild  spacetime admits two foliations by lightlike hypersurfaces. In fact, for every lightlike vector field $\xi$ in the $tr$-half plane, the distribution given by the vector fields $\widetilde{g}$-orthogonal to $\xi$ is integrable and every leaf inherits a degenerate metric from $\widetilde{g},$ Lemma \ref{190723A}.

\item The vector field $\partial t$ is a timelike Killing vector field. That is,  the exterior Schwarzschild  spacetime is static.
\end{enumerate}

\noindent These properties rely on the following facts.  The metric $\widetilde{g}$ is a warped product metric given by a Lorentzian metric on an open subset of the  $tr$-plane and a Riemannian metric,  \cite[Chap. 7]{One83}. Moreover, the metric on the $tr$-plane admits a globally defined lightlike vector field  and the function $
f^2(r):=\Big(1- \frac{2\mathbf{M}}{r^{m-1}} \Big)
$
does  depend only on $r$.

These facts lead us to consider the  following class of Lorentzian warped product manifolds.   

\begin{definition}\label{190723B}

{\rm
 
A Lorentzian warped product manifold $(\widetilde{M}, \widetilde{g})=B\times_{\lambda} F$   is said to be an $(m+2)$-dimensional generalized (exterior) Schwarzschild spacetime when $B$ is an open subset of $\R^{2}$ with canonical coordinates $(t,r)$ and metric
\begin{equation}\label{17112022A}
g_B=-f^2(r)dt^2+\dfrac{1}{f^2(r)}dr^2,
\end{equation}
where $f(r)> 0$,  $(F,g_F)$ is a $m$-dimensional connected Riemannian manifold and $\lambda\in C^{\infty}(B)$ with $\lambda>0$ is the warping function. That is, $\widetilde{M}= B \times F$ and 
$$\widetilde{g}=\pi_B^*(g_B)+(\lambda\circ\pi_B)^2\pi_F^*(g_F),$$ where $\pi_B$ and $\pi_F$ are the natural projections on $B$ and $F$, respectively \cite[Chap. 7]{One83}.}
\end{definition}
Note that the Minkowski spacetime can be described in two ways from this setting. Namely, as was mentioned above  and  for $f(r)^{2}=1$, $\lambda=1$ and $F=\E^{m}$. 

For $\lambda(t,r)=r$ and $F=\S^{m}$, this class  includes  relevant spacetimes with spherical symmetry.  
Namely,  we set 
\begin{equation}\label{app}
f^{2}(r)=1-\frac{2 \mathbf{M}}{r^{m-1}}+ \frac{q^2}{r^{2m-2}}- \frac{2 \Lambda r^2}{m(m+1)},
\end{equation}
where  $\mathbf{M}$ is called the mass parameter, $q$ is the charge and $\frac{2 \Lambda r^2}{m(m+1)}$
is the cosmological constant when positive and $\frac{2 \Lambda r^2}{m(m+1)}=-1/R^2$ when negative, where $R$ is the anti-de Sitter radius. For $q= \Lambda=0$, we get the Schwarzschild metric and
for $q \neq 0$, $\Lambda= 0$,  the Reissner-Nordstr\"om metric with total charge $q$.
The  de-Sitter and anti-de-Sitter versions correspond to $\Lambda >0$ and  $\Lambda <0 $, respectively.

Karl Schwarzschild discovered  in 1916 the point-mass solution to Einstein  equations that bears his name.  Historically,   this solution was the first and more important nontrivial solution of the vacuum Einstein equations.  This is the reason why we have named this class of  spacetimes as  generalized Schwarzschild spacetimes.

The  generalized Schwarzschild spacetimes  admit two lightlike vector fields $\xi, \eta \in \mathfrak{X}(B)\subset \mathfrak{X}(\widetilde{M})$,  (\ref{11102022B}).  The distributions $D_{\xi}$ and $D_{\eta}$ defined by the $\widetilde{g}$-orthogonal vector fields to $\xi$ and $\eta$, respectively, are involutive,  Lemma \ref{190723A}.
Therefore, we have two transverse foliations by lightlike hypersurfaces of $\widetilde{M}$.  It is worth pointing out that   the vector field $\partial t\in \mathfrak{X}(\widetilde{M})$ is Killing if and only if the warping function $\lambda$ depends only on the (radial) coordinate $r$.  
The integral curves of $\xi$ and $\eta$ are called lightlike geodesic generators of the corresponding lightlike hypersurfaces.  
Recall that we can scale $\xi$ and $\eta$ so that they are geodesic vector fields, \cite{Gall}.
As was mentioned, the Minkowski spacetime $\L^{m+2}$ admits  two descriptions as generalized Schwarzschild spacetime.  Every description provides different foliations by lightlike hypersurfaces,  see details in Example \ref{280723A}.

\smallskip

The main aim of this paper is the study of $n$-dimensional spacelike submanifolds $M$ immersed in  generalized Schwarzschild spacetimes.
The research on spacelike submanifolds has been developed both from  physical and geometric interest.  For instance, the Cauchy problem for the Einstein equations
is formulated as an initial data problem on a Riemannian manifold which becomes a Cauchy hypersurface in the solution spacetime, see \cite[Chap. 7]{Lee}.
Recall also  the  Penrose incompleteness theorem which relates the existence of a trapped codimension two spacelike submanifold with the singularities of certain spacetimes \cite[Chap. 7]{Lee}. The notion of trapped submanifold is usually given in terms of the mean curvature vector field of the submanifold,  Definition \ref{110823C}.

Most of our results are focused on the particular situation in which the spacelike submanifold $M$ is contained in a leaf of the above mentioned foliations by lightlike hypersurfaces.
As it is well-known, lightlike hypersurfaces inherits a degenerate metric from the Lorentzian ambient metric and play an important role in General Relativity  as event horizons of black holes \cite{Gourgoulhon}.  The classical theory of submanifolds fails for these hypersurfaces.   We think  that the study of  codimension two spacelike submanifolds which factor through  a lightlike hypersurface can provide a tool to understand the geometry of such hypersurfaces. 
A different approach to lightlike manifolds from the Cartan geometries  has been developed in \cite{Pal21}.
The study of  codimension two spacelike submanifolds in lightlike hypersurfaces has been previously developed in \cite{PaPaRo}, \cite{PaRoRo} and \cite{PaRo}  for the case of compact submanifolds into the lightlike cone in the Minkowski spacetime.   The non-compact case is considered in \cite{ACR2}  and the study of trapped submanifolds into lightlike hypersurfaces of the  Sitter spacetime appears in \cite{ACR1}.  This approach has been also applied to Brinkmann spacetimes. Recall that  Brinkmann spacetimes admit a parallel lightlike vector field and then, they have a foliation by lightlike hypersurfaces.  Spacelike submanifolds which lie in such hypersurfaces have been studied in \cite{CaPaRo} for the compact case and in \cite{PPR} for more general settings.

We would like to emphasize that  the study of spacelike submanifolds in generalized Friedmann-Lema\^itre-Robertson-Walker spacetimes goes back to the seminal work \cite{AlRoSan}.  Since then multiple researchers  have developed  this topic. These spacetimes are written as $I\times_{\lambda}F$ with metric $-dt^2 + \lambda^2(t) F$. 
From this point of view, the submanifolds in generalized Schwarzschild spacetimes can be seem  as the  next natural step to shed light in the theory of spacelike submanifolds. As far as we know, there is no many works devoted to this problem.  For instance, in the setting of stationary spacetimes,  the study of  prescribed mean curvature problem in Schwarzschild and  Reissner-Nordstr\"om spacetimes appears in \cite{DRT}. On the other hand, the results in \cite{WWZ} have been enriching and have given us a better approach  for the development of our paper.

\smallskip

The plan of the paper is as follows. Section 2 exhibits basic notions, notations and also includes several technical results to be used later.
Section 3 focuses on the class of  Lorentzian manifolds  we are interested in, Definition \ref{190723B}.  
Since this class of Lorentzian manifolds are warped product manifolds, we particularize the formulas for theirs Levi-Civita connections from \cite[Chap. 7]{One83}.
For a spacelike immersion in a generalized Schwarzschild spacetime  $\Psi:M\rightarrow B\times_{\lambda} F,$ we have written   $\Psi=(\Psi_B,\Psi_F)$,   $u:=t\circ \Psi_{B}$ and  $v:=r\circ \Psi_{B}.$  
 Lemma \ref{020323A} states  that a spacelike immersion factors through an integral hypersurface of $D_{\xi}$  if and only if
$
\nabla v=(f\circ \Psi_{B})^2 \nabla u ,
$
where $\nabla$ denotes the gradient operator corresponding to the induced metric $g$ on  $M$. In order to make the presentation of the results more fluid, in the Introduction we specialize  our results and discussions to the  distribution $D_{\xi}$ and its integral lightlike hypersurfaces.  Almost all the results admit a similar version for the other distribution $D_{\eta}$.

Section 4 exhibits  fundamental equations for spacelike immersions in generalized Schwarzschild  spacetimes.
As a consequence, we obtain an  integral characterization of compact spacelike immersions through leaves of $D_{\xi}$,
Theorem \ref{070323C}.
\begin{quote}
Assume $\Psi:M\rightarrow B\times_{\lambda}F$ is a compact spacelike immersion in a generalized Schwarzschild spacetime  with $f'>0$ (resp. $f'<0$). Then 
$$
\int_{M} \Big[ n\,\widetilde{g}(\mathbf{H},\xi^{\bot}) +\Big(\dfrac{\xi\lambda}{\lambda}\circ \Psi_{B}\Big)\left[n+2g(\xi^{\top}, \eta^{\top})\right]\Big]\, d\mu_{g} \geq 0.\quad (\textrm{resp.} \leq 0),
$$
where the superscripts $\top$ and $\bot$ denote the tangent  and normal parts of the indicated vector fields, respectively.
The equality holds if and only if $M$ factors through an integral hypersurface  of $D_{\xi}.$
\end{quote}

Our main results are in Sections 5, 6 and 7, where we will focus on the case of spacelike submanifolds factoring through a lightlike integral hypersurface of $D_{\xi}$ or $D_{\eta}$. First, in
Section 5 we deal with the case of arbitrary codimension.  Our maim aim here is to find  several conditions  which assure that the immersion  factors through   a slice of the generalized Schwarzschild  spacetime.  That is, the functions $u$ and $v$ are constants. Thus, the  setting  is a spacelike immersion $\Psi:M\rightarrow B\times_{\lambda} F$  through an integral hypersurface of $D_{\xi}$.  
\begin{quote}
\begin{itemize}

\item Assume $\eta \lambda$ signed and $M$ compact with $\mathbf{H}=0$. Then $M $ factors through a slice  and the immersion of $M$ in such slice is minimal, Corollary \ref{020623A}.

\item Assume $M$ compact with $\mathrm{Ric}^g(\nabla v, \nabla v)\leq 0$. The normal vector field $\eta^{\bot}$ is an umbilical direction if and only if $M $ factors through a slice,  Theorem \ref{20032023A}.

\item Assume $\xi \lambda \neq 0$ never vanishes. Then $M$ factors through a slice if and only if $\nabla^{\bot}\xi=0$, Theorem \ref{250523A}.

\end{itemize}

\end{quote}

\noindent The assumptions in Corollary \ref{020623A} and Theorem \ref{250523A} hold for a wide family of generalized Schwarzschild spacetime. For Theorem \ref{20032023A}, it is a key fact that if $\eta^{\bot}$ is an umbilical direction then $\nabla v$ is a conformal vector field (\ref{22032023A}).

 In our notion of generalized Schwarzschild spacetime, the geometry of the Riemannian part $F$ is arbitrary. Nevertheless, the case of  spherical symmetry
is the most relevant from the physical point of view.  If we assume $M$  compact and $\eta^{\perp}$ an umbilical direction, Theorem \ref{260723A} states a condition on the Ricci tensor which shows that, when $\nabla v$ is a nonzero vector field, $M$ is isometric to a sphere $\S^{n}(c)$ of constant sectional curvature $c$. This result is a consequence of  \cite[Theor. 1]{DAL}.
Therefore,  in case that the codimension of $M$ is two and $F$ is simply-connected, by means of Proposition \ref{020623B},  the manifold $F$ must be a topological sphere.

Section 6 is devoted to study codimension two ($n=m$) immersions through these lightlike integral hypersurfaces.  In the terminology of black holes,  a such immersion $M$ is called a cross-section when every lightlike geodesic generator  intersects $M$ at most one,   \cite{Gourgoulhon}.
At topological level, every codimension two immersion through a lightlike integral hypersurface is a covering space of the fiber $F$ but not necessarily a Riemannian covering,  Proposition \ref{020623B} and Remark \ref{140823A}.  The mean curvature vector field of these immersions is obtained  in Proposition \ref{010623B} and Corollary \ref{280723D} as follows
\begin{quote}
$$
\mathbf{H}=\left[\dfrac{\eta\lambda}{\lambda}\circ \Psi_{B}-\Big(\dfrac{\xi\lambda}{2\lambda}\circ \Psi_{B}\Big)\|\nabla v\|^2+\dfrac{1}{m}\Delta v\right]\xi+\Big(\dfrac{\xi\lambda}{\lambda}\circ \Psi_{B}\Big)\ell^{\xi},
$$
where $\ell^{\xi}$ is the normal lightlike vector field to $M$ with $\widetilde{g}(\xi, \ell^{\xi})=-1$.  Furthermore, we  compute that
$$
\|\mathbf{H}\|^{2}=\frac{1}{v^2}\left( (f\circ \Psi_{B})^2- \frac{S^{\Psi_{F}^{*}(g_{F})}- v^2  S^{g}}{m(m-1)}\right),
$$
where $S^{g}$ and $S^{\Psi_{F}^{*}(g_{F})}$ are the scalar curvatures of the induced metric $g$ and $\Psi_{F}^{*}(g_{F})$ on $M$, respectively.
\end{quote}
These results extend  previous ones in \cite{ACR1}, \cite{ACR2}, \cite{PaPaRo}  and \cite{PaRo}, see details in Remark \ref{140823B}.
The second formula shows a relation between the intrinsic and extrinsic geometry of the codimension two immersions through  lightlike integral hypersurfaces. A such kind of relation has been previously pointed out for the case of the lightlike cone in the Minkowski spacetime in \cite{PaPaRo}  and \cite{PaRo}. Section 6 also contains a characterization of marginally trapped immersions when the warping function $\lambda$ agrees with the radial coordinate,   Corollary \ref{140823C}.
This result extends \cite[Cor. 6.3]{ACR2} where the case of the lightlike cone in the Minkowski spacetime was studied.

In Section 7 we proceed with the study of  immersions with parallel mean curvature vector field. That is, we consider the condition $\nabla^{\perp}\mathbf{H}=0.$
Under a technical condition,  Theorem  \ref{130623B}  provides an intrinsic characterization of the slices as the unique codimension two immersions through  lightlike integral hypersurfaces with parallel mean curvature vector field.

\section{Preliminaries}\label{Prel}

\noindent Let $(\widetilde{M}, \widetilde{g})$ be an $(m+2)$-dimensional Lorentzian manifold. A smooth immersion  $\Psi:M\rightarrow (\widetilde{M},\widetilde{g})$ of a (connected) $n$-dimensional  manifold $M$  is said to be spacelike when the induced metric $g:= \Psi^{*}(\widetilde{g})$ is Riemannian.  We assume $n\geq 2$ along this paper.
For any point $x\in M$, we denote $\Psi(x)\in \widetilde{M}$ by the same letter $x$ if there is no danger of confusion. We also ignore the differential map of the immersion $\Psi$. Thus $T_{x}M$ is considered as subspace of $T_{x}\widetilde{M}$ and its orthogonal complement $T_{x}M^{\perp} \subset T_{x}\widetilde{M}$ is called the normal space of $M$ at $x$. 
We write $\overline{\mathfrak{X}}(M)$ for the $C^{\infty}(M)$-module of vector fields along the immersion $\Psi$. 
The set of vector fields $\mathfrak{X}(M)$ may be seen as a $C^{\infty}(M)$-submodule of $\overline{\mathfrak{X}}(M)$ in a natural way and for every $E\in \mathfrak{X}(\widetilde{M})$ we have  its restriction  $E\mid_{\Psi} \in \overline{\mathfrak{X}}(M)$. For $V\in \overline{\mathfrak{X}}(M)$, we have the orthogonal decomposition 
$$V_x= V_x^{\top}+V_x^{\bot},$$ where $V^{\top}_{x}\in T_{x}M$ and $V^{\bot}_{x}\in T_{x}M^{\bot}$ for all $x\in M$.
The vector fields  $V^{\top}$ and $V^{\bot}$ are called the tangent part of $V$ and  the normal part of $V$, respectively. 
The $C^{\infty}(M)$-submodule of $\overline{\mathfrak{X}}(M)$ of all normal vector fields along $\Psi$ is denoted by $\mathfrak{X}^{\perp}(M)$, that is, $$\mathfrak{X}^{\perp}(M)=\{V \in \overline{\mathfrak{X}}(M): V^{\top}=0\}.$$

Now, let us write $\nabla$ and $\widetilde{\nabla}$ for the Levi-Civita connections of $M$ and $\widetilde{M}$, respectively. 
As usual, we also denote by $\widetilde{\nabla}$ the induced connection and  by $\nabla^{\perp}$ the normal connection on $M$. 
The decomposition of the induced connection $\widetilde{\nabla}$, into tangent and normal parts, leads to the Gauss and Weingarten formulas of $\Psi$ as follows
\begin{equation}\label{shape}
\widetilde{\nabla}_VW= \nabla_VW + \mathrm{II}(V,W) \quad \quad \mathrm{and} \quad \quad \widetilde{\nabla}_V\zeta=-A_{\zeta}V+\nabla^{\perp}_V\,\zeta,
\end{equation}
for every tangent vector fields $V,W\in\mathfrak{X}(M)\subset \overline{\mathfrak{X}}(M)$ and $\zeta\in\mathfrak{X}^{\perp}(M)$. Here $\mathrm{II}$ denotes the second fundamental form and $A_{\zeta}$ the shape operator (or Weingarten endomorphism) associated to $\zeta$. 
Every shape operator $A_{\zeta}$ is self-adjoint and the second fundamental form is symmetric, they are also related by the following formula
\begin{equation}\label{230321C}
g\left(A_{\zeta}V,W\right)  = \widetilde{g} \left(\mathrm{II}(V,W), \zeta \right).
\end{equation}
The mean
curvature vector field  is defined by
$\mathbf{H}=\frac{1}{n}\,\mathrm{trace}_{g}\mathrm{II}.$ From (\ref{230321C}) we have
$$
\mathrm{trace}(A_{\zeta})=n\widetilde{g}(\mathbf{H},\zeta).
$$

\bigskip

Although, we are interested here in the class of spacetimes given in Definition \ref{190723B}, there are several properties which can be stated in a more general setting.
Let $(B,g_B)$ be a two dimensional  oriented Lorentzian manifold and $(F,g_F)$ a $m$-dimensional connected Riemannian manifold. Fix $\lambda\in C^{\infty}(B)$ with $\lambda>0$,  we are interested in the Lorentzian warped product manifold  given by the   product manifold $\widetilde{M}=B\times F$ endowed with the Lorentzian metric 
$$\widetilde{g}=\pi_B^*(g_B)+(\lambda\circ\pi_B)^2\pi_F^*(g_F),$$ where $\pi_B$ and $\pi_F$ are the natural projections on $B$ and $F$, respectively \cite[Chap. 7]{One83}.
As usual, we denote the Lorentzian manifold  $(\widetilde{M},\widetilde{g})$ as $B\times_{\lambda} F$ and $\lambda$ is called the warping function.
 The set of vector fields $ \mathfrak{X}(B)$ and $\mathfrak{X}(F)$ can be lifted to $\mathfrak{X}(\widetilde{M})$. Typically, we use the same notation for a vector field and its lift and then,
 every vector field $E\in \mathfrak{X}(\widetilde{M})$ has a unique expression as $E=X+V$ where $X\in \mathfrak{X}(B)$ and  $V\in \mathfrak{X}(F)$.

For our aims here, we assume  there exists a global lightlike vector field $\xi\in\mathfrak{X}(B)\subset\mathfrak{X}(\widetilde{M})$. That is, we have $g_{B}(\xi, \xi)=0$ and $\xi_{x}\neq 0$ for all $x\in B$. 
Taking into account that $\mathrm{dim}\, B=2$, it is not difficult to show that there exists a $1$-form  $\alpha\in \Omega^{1}(B, \R)$ such that
\begin{equation}\label{250223B}
\nabla^{B}\xi=\alpha \otimes \xi\end{equation}
where  $\nabla^{B}$ is the Levi-Civita connection of $B$. The assumption on the existence of the vector field $\xi$ has the following key consequence. 

\begin{lemma}\label{190723A}
The distribution
$D_{\xi}=\{E\in\mathfrak{X}(\widetilde{M}):\widetilde{g}(E,\xi)=0\}$ on $\widetilde{M}$
is involutive. 
\end{lemma}
\begin{proof} For  $X+V, Y+ W \in D_{\xi}$ a straightforward computation gives
$$
\widetilde{g}([X+V, Y+ W], \xi)=g_{B}([X,Y], \xi)\circ \pi_{B}.
$$
Taking into account that $X+V\in D_{\xi}$  if and only if $g_{B}(X,\xi)=0$, we obtain from (\ref{250223B}) that
$$g_{B}([X,Y], \xi)=-g_{B}(\nabla^{B}_{X}\xi, Y)+ g_{B}(\nabla^{B}_{Y}\xi, X)=0.$$ 
\end{proof}

Therefore,  through every point  $(x,p)\in \widetilde{M}$ passes  a maximal integral submanifold $\mathcal{L}$ of the distribution $D_{\xi}$ and we have a foliation of the manifold $\widetilde{M}$ by hypersurfaces. 
If we write $\gamma\colon I\to B$ for the maximal integral curve of the vector field $\xi$ with initial condition $\gamma(0)=x\in B$, then the hypersurface $\mathcal{L}$  is given by
$$
\mathcal{L}=\{(\gamma(t), p)\in B\times F: t\in I\, \, , p\in F\}.
$$
$\mathcal{L}$ inherits a degenerate metric tensor from $\widetilde{g}$ whose radical is spanned by the vector field $\xi\mid_{\mathcal{L}}$. 
Every  smooth section $\sigma= (\sigma_{B}, \mathrm{Id}_{F})$ of the natural projection $\mathcal{L}\to F$ provides a spacelike immersion in $\widetilde{M}$ with induced metric $g=\sigma^{*}(\widetilde{g})=(\lambda\circ \sigma_{B})^{2}g_{F}$. Hence the induced metric $g$ belongs to the same conformal class of $g_{F}$. In particular, $\mathcal{L}$ is a subset of the bundle of scales of $F$ for the conformal class of $g_{F}$ \cite[Chap. 1]{CS09}.

\begin{remark}
{\rm The projection $B\times_{\lambda}F \to B$ is a semi-Riemannian submersion \cite{One66}  (see also \cite[Chap. 7]{One83}).   In the terminology of semi-Riemannian submersions, every maximal integral submanifold $\mathcal{L}$ of the distribution $D_{\xi}$ is the horizontal lift of a maximal integral curve of the vector field $\xi$ on $B$.
}
\end{remark}

Now let $\Psi:M\rightarrow B\times_{\lambda} F$ be an arbitrary  spacelike immersion. The immersion $\Psi$ can be written  $$\Psi=(\Psi_B,\Psi_F),$$ where $\Psi_B= \pi_{B}\circ \Psi$ and $\Psi_F = \pi_{F}\circ \Psi$.
For the vector field  $ \xi |_{\Psi}$ we have 
$$
0=\widetilde{g}(\xi |_{\Psi},\xi |_{\Psi})=\widetilde{g}(\xi^{\top},\xi^{\top})+\widetilde{g}(\xi^{\perp},\xi^{\perp}).
$$
Hence, $\xi^{\perp}$ does not vanish at any point of $M$ and  $\widetilde{g}(\xi^{\perp}, \xi^{\perp})\leq 0$, in other words, $\xi^{\perp}$ is a causal normal vector field. On the other hand, the vector field $\xi^{\top}$ vanishes identically if and only if $M\subset \mathcal{L}$. In this case, we say that $M$ factors through the  integral hypersurface $\mathcal{L}$ of $D_{\xi}$.

\begin{remark}\label{300523A}
{\rm
Let recall that the two dimensional manifold $B$ is assumed to be orientable. Thus, the existence of the lightlike vector field $\xi$ implies that there is another  lightlike vector field $\eta \in \mathfrak{X}(B)$ which is uniquely determined by the normalization condition $g_{B}(\xi, \eta)=-1$.  As for $\xi$, we have $\nabla^{B}\eta= -\alpha \otimes \eta$ and the corresponding distribution $D_{\eta}$ is also involutive.
Every each maximal integral submanifold  $\mathcal{N}$ of  $D_{\eta}$ inherits a degenerate metric tensor from $\widetilde{g}$ whose radical is now spanned by the restriction of $\eta$ to $\mathcal{N}.$
}
\end{remark}

In order to be used later, let us recall the notion of parabolic Riemannian manifold.
A (non necessary complete) Riemannian 
manifold is parabolic if the only subharmonic functions
bounded from above that it admits are the constants. That is, a  Riemannian manifold $M$ is parabolic when $\Delta v \geq 0$ and $\sup_{M} v< +\infty$ for a smooth function $v\in C^{\infty}(M)$ implies $v$ must be constant
(see, for instance, \cite{Ka} and \cite{Gr}). From a 
physical point of view,   the parabolicity  
is equivalent to the recurrence of the Brownian motion on a Riemannian 
manifold \cite{Gr}.  Let us also recall that every 
complete Riemannian surface with non-negative 
Gaussian curvature is parabolic \cite{H}.  Even more,  every complete Riemannian surface with finite total curvature is parabolic \cite{H}.
In arbitrary 
dimension there is no clear relation between parabolicity
and sectional curvature.   Nevertheless, there exist sufficient conditions to 
ensure the parabolicity of a Riemannian manifold of arbitrary dimension 
based on the volume growth of its geodesic balls  \cite{AMR}.

\section{Generalized Schwarzschild  spacetimes}

From now on $ B\times_{\lambda} F$ is a generalized Schwarzschild  spacetime, Definition \ref{190723B}.  We can give two lightlike vector fields in $B$ as follows
\begin{equation}\label{11102022B}
\xi=\dfrac{1}{f^2}\partial t+\partial r\,\text{ and }\,\eta=\frac{1}{2}(\partial t-f^2\partial r)
\end{equation}
with $\widetilde{g}(\xi,\eta)=-1.$  
The Levi-Civita connection  $\nabla^B$ is directly computed from (\ref{17112022A}) as follows 
$$
\nabla^B_{\partial t}\partial t=f^3f'\partial r,\quad\quad \nabla^B_{\partial t}\partial r=\dfrac{f'}{f}\partial t\quad\quad\text{and}\quad\quad \nabla^B_{\partial r}\partial r=-\dfrac{f'}{f}\partial r
$$
where $f'=\partial_{r}f$. The Gauss curvature  of the metric $g_B$ is 
$
K^{B}=-\left((f')^2 + ff''\right).
$
 The $1$-form defined in (\ref{250223B}) satisfies
\begin{equation}\label{15112022C}
\alpha= f'\left(fdt-\dfrac{1}{f}dr\right).
\end{equation}
In particular, $\alpha(\xi)= \alpha(\eta)=0$ and then $\xi$ and $\eta$ are geodesic vector fields.

Let us recall that the Levi-Civita connection of $\widetilde{g}$ is given in \cite[Prop. 7.35]{One83} as follows. For  $X,Y\in \mathfrak{X}(B)\subset \mathfrak{X}(\widetilde{M})$ and  $V,W\in \mathfrak{X}(F)\subset \mathfrak{X}(\widetilde{M})$, we have
\begin{equation}\label{250223A}
\widetilde{\nabla}_{X}Y=\nabla^{B}_{X}Y, \quad \widetilde{\nabla}_{X}V=\widetilde{\nabla}_{V}X= \frac{X\lambda}{\lambda}V,\quad 
\widetilde{\nabla}_{V}W=-\frac{\widetilde{g}(V,W)}{\lambda}\nabla^{B}\lambda+ \nabla^{F}_{V}W,
\end{equation}
where  $\nabla^{B}$ and $\nabla^{F}$ are the Levi-Civita connections of $B$ and $F$, respectively. For every $h\in \mathcal{C}^{\infty}(B)$, we write $\nabla^{B}h$ for the gradient of $h$  with respect to the metric $g_{B}$. Besides  $\widetilde{\nabla}(h\circ \pi_{B})= \nabla^{B}h \circ \pi_{B}$, \cite[Lemma. 7.34]{One83}. Straightforward computations show for the natural coordinates $t$ and $r$ on $B$ and for the warping function $\lambda$ that
$$
\nabla^B t=-\dfrac{1}{f^2}\partial t,\quad \nabla^B r=f^2\partial r,\quad \nabla^B f=f^2f'\partial r\quad\text{and}\quad \nabla^B\lambda=-\dfrac{\lambda_t}{f^2}\partial t+f^2\lambda_r\partial r.
$$

\begin{remark}
{\rm Let $\mathcal{L}$ be an integral hypersurface of the distribution $D_{\xi}$. Recall that the null-Weingarten map  $b_{\xi}$ is defined for every point $p\in \mathcal{L}$ as
$$
b_{\xi}\colon T_{x}\mathcal{L}/\xi_{x}\to T_{x}\mathcal{L}/\xi_{x}, \quad [w]\mapsto [\widetilde{\nabla}_{w}\xi],
$$
where $[\,\,]$ denotes the class in the quotient vector space $T_{x}\mathcal{L}/\xi_{x}$,  see \cite{Gall}.  The lightlike manifold $\mathcal{L}$ is said to be totally geodesic when $b_{\xi}=0$. A direct computation from (\ref{250223A}) 
gives $b_{\xi}([w])= \frac{\xi \lambda}{\lambda}[w]$.
This formula implies that $\mathcal{L}$ is a totally geodesic lightlike hypersurface if and only if $\xi\lambda=0$. In a similar way, the integral hypersurfaces of $D_{\eta}$ are totally geodesic lightlike hypersurfaces if and only if $\eta\lambda=0$. 
}
\end{remark}

\begin{remark}
{\rm 
Lorentzian manifolds  admitting a global  parallel and lightlike vector field $\xi$ were introduced in \cite{B}. Such a Lorentzian manifolds are  called  Brinkmann spacetimes  Hence, a generalized Schwarzschild spacetime is a Brinkmann spacetime if and only if  $\alpha=0$.
Recall that a Lorentzian manifold $(\bar{M},\bar{g})$ is said to be static when admits a Killing vector field $K$ with $\bar{g}(K,K)<0$. The timelike vector field $\partial_{t}$ in a generalized Schwarzschild spacetime is Killing if and only if $\lambda_{t}=0.$ This is the case of the classical Schwarzschild spacetime.
}
\end{remark}

\begin{remark}
{\rm In a general setting, given a semi-Riemannian manifold $(M,g)$, a vector field $X\in \mathfrak{X}(M)$  is said to be recurrent when there is a $1$-form $\alpha$ on $M$ such that   $\nabla X=\alpha \otimes X$, where $\nabla$ is the Levi-Civita connection of $g$.
In particular, equations (\ref{250223B}) and (\ref{250223A}) imply that the vector fields $\xi$ and $\eta$ are recurrent.  This  property widely generalizes the Brinkmann spacetimes.  Lorentzian manifolds with recurrent lightlike vector fields have been studied in \cite{Leistner}.

}
\end{remark}

For $\Psi:M\rightarrow B\times_{\lambda} F$ a spacelike immersion in a generalized Schwarzschild spacetime, we have the smooth functions on $M$ given by
$
u=t\circ \Psi_{B}$ and  $v=r\circ \Psi_{B}.$
We have for the gradients of these functions with respect to the induced metric $g$ that
\begin{equation}\label{270223D}
\nabla u=-\frac{1}{(f\circ \Psi_{B})^2}\partial t^{\top}\quad\text{and}\quad \nabla v=(f\circ \Psi_{B})^2\partial r^{\top}
\end{equation}
and therefore
\begin{equation}\label{21112022A}
\xi^{\top}=\dfrac{1}{(f\circ \Psi_{B})^2}\nabla v-\nabla u\quad\text{and} \quad \eta^{\top}=-\dfrac{1}{2}\left(\nabla v+(f\circ \Psi_{B})^2\nabla u\right).
\end{equation}
These formulas (\ref{21112022A}) lead to the following characterization for spacelike immersions through integral submanifolds of the distributions $D_{\xi}$ or $D_{\eta}.$

\begin{lemma}\label{020323A}
A spacelike immersion $\Psi:M\rightarrow B\times_{\lambda} F$  in a generalized Schwarzschild  spacetime factors through an integral hypersurface  $\mathcal{L}$ (resp. $\mathcal{N}$) of the distribution $D_{\xi}$ $($resp. $D_{\eta})$ if and only if
$$
\nabla v=(f\circ \Psi_{B})^2 \nabla u \quad (\mathrm{resp.}\, \nabla v=- (f\circ \Psi_{B})^2 \nabla u ).
$$
\end{lemma}

The spacelike immersion
$
F\hookrightarrow  B\times_{\lambda} F $ given by $ x\mapsto (t_{0}, r_{0}, x)
$
is called  the slice at level $(t_{0}, r_{0}).$ From \cite[Prop.  7.35]{One83},   its mean curvature vector field is
\begin{equation}\label{110923A}
    \mathbf{H}= -\frac{\nabla^{B}\lambda}{\lambda}(t_{0}, r_{0})
\end{equation}
 and for  $\lambda(t,r)=r$, this formula reduces to
$
\mathbf{H}= -\frac{f^{2}(r_{0})}{ r_{0}}\partial r\mid_{(r_{0}, t_{0})}.
$
Recall that  slices are totally umbilical spacelike embedded submanifolds, \cite[Chap. 7]{One83}  and
note that every  spacelike immersion which factors through an integral hypersurface of $D_{\xi}$ and, at the same time, through an integral hypersurface of $D_{\eta}$ must factors through a slice.

\begin{example}\label{280723A}
{\rm
As was mentioned in the Introduction section, the $(m+2)$-dimensional Minkowski spacetime $\L^{m+2}$ can be described in two ways as a generalized Schwarzschild  spacetime.
The first one is $B=\R^{2}$, $f(r)=1$, $\lambda(t,r)=1$ and $(F,g)=\E^{m}$. The lightlike vector fields in (\ref{11102022B}) are
$
\xi=\partial t+ \partial r
$
 and  
$\eta=\frac{1}{2}(\partial t-\partial r).
$
Hence, the leaves of  the lightlike foliations are lightlike hyperplanes in $\L^{m+2}.$
The second one is obtained by taking $B=\R \times \R_{+}$, $f(r)=1$, $\lambda(t,r)=r$ and $(F,g)=\S^{m}$. The lightlike vector fields in (\ref{11102022B}) are
$
\xi=\partial t + \partial r
$
 and  
$\eta=\frac{1}{2}(\partial t-\partial r).
$
The smooth map
\begin{equation}\label{16082023-1}
\varphi:(\R \times \R_{+})\times_{r} \S^{m}\to \L^{m+2}, \quad (t,r, x)\mapsto (t, rx)
\end{equation}
provides an isometry on the open subset $\{(t, p)\in \L^{m+2}: p\neq 0\}$. The  foliations by lightlike hypersurfaces given in Lemma \ref{190723A}
correspond via this isometry with the lightlike cones with vertex at the points $(t,0)\in \L^{m+2}.$

}
\end{example}

\section{Immersions in  generalized Schwarzschild spacetimes}

Along this section  $$\Psi:M\rightarrow B\times_{\lambda} F$$ is a fixed spacelike immersion which does not necessary factors through an integral submanifold of $D_{\xi}$ or $D_{\eta}$. 
For every vector field $V\in \mathfrak{X}(M)$ and $x\in M$, we denote 
$$V^{B}_{x}= T_{x}\Psi_{B}\cdot V_{x} \quad \textrm{and }\quad V^{F}_{x}= T_{x}\Psi_{F}\cdot V_{x}.$$ Since we  agree to ignore the differential map of $\Psi$, this means that
$
V=V^{B}+ V^{F}.
$
We get that 
\begin{equation}\label{090523A}
V^{B}=g(V, \nabla u)\partial t+ g(V, \nabla v)\partial r,
\end{equation}
and  from (\ref{270223D}), we have
\begin{equation}\label{17112022E}
(V^{B})^{\top}=-(f\circ \Psi_{B})^2V(u)\nabla u + \dfrac{1}{(f\circ\Psi_{B})^2}V(v)\nabla v.
\end{equation}
As a consequence of  (\ref{250223A}), we get
$$
\widetilde{\nabla}_{ V}(\xi\mid_{\Psi})=\alpha(V^{B})\xi\mid_{\Psi}+\Big(\dfrac{\xi\lambda}{\lambda}\circ \Psi_{B}\Big) V^F. 
$$
The Gauss and Weingarten formulas (\ref{shape}) imply that
\begin{equation}\label{17112022B}
\nabla_V\xi^{\top}+\mathrm{II}(V,\xi^{\top})- A_{\xi^{\bot}}V+\nabla^{\perp}_V\,\xi^{\bot}=\alpha(V^{B})\xi\mid_{\Psi}+\Big(\dfrac{\xi\lambda}{\lambda}\circ \Psi_{B}\Big) V^F.
\end{equation}
In particular for the tangent parts to $M$, we get 
\begin{equation}\label{270223B}
\nabla_V\xi^{\top}-A_{\xi^{\bot}}V=\alpha(V^{B})\xi^{\top}+\Big(\dfrac{\xi\lambda}{\lambda}\circ \Psi_{B}\Big)( V^{F})^{\top}.
\end{equation}
From  (\ref{17112022E}) and taking into account that
$
( V^{F})^{T}= V-( V^{B})^{T}
$, the equation (\ref{270223B}) reduces to
$$
\nabla_V\xi^{\top}-A_{\xi^{\bot}}V=\alpha(V^{B})\xi^{\top}+\Big(\dfrac{\xi\lambda}{\lambda}\circ \Psi_{B}\Big)\Big(V+(f\circ \Psi_{B})^2V(u)\nabla u - \dfrac{1}{(f\circ\Psi_{B})^2}V(v)\nabla v\Big)
$$
and then,
\begin{equation} \label{17112022G}
\begin{split}
\mathrm{div}(\xi^{\top}) - n\,\widetilde{g}(\mathbf{H},\xi^{\bot}) & = \Big(\dfrac{\xi\lambda}{\lambda}\circ \Psi_{B}\Big)\left[n+(f\circ \Psi_{B})^2\|\nabla u\|^2 -\dfrac{1}{(f\circ \Psi_{B})^2}\|\nabla v\|^2\right] \\
 & + (\Psi_{B}^{*}\alpha) (\xi^{\top}). 
\end{split}
\end{equation}
Note that $(\Psi_{B}^{*}\alpha) (\xi^{\top})=\alpha((\xi^{\top})^{B})$. A straightforward computation from (\ref{21112022A}) and (\ref{090523A}) gives 
$$
(\xi^{\top})^{B}=\Big(\frac{1}{(f\circ \Psi_{B})^2}g(\nabla u, \nabla v)- \| \nabla u\|^2\Big)\partial t +\Big(\frac{1}{(f\circ \Psi_{B})^2}\| \nabla v\|^2 -g(\nabla u, \nabla v) \Big)\partial r
$$
and then, from (\ref{15112022C}), we have
$$
(\Psi_{B}^{*}\alpha) (\xi^{\top})= - (ff'\circ \Psi_{B}) \| \xi^{\top}\|^{2}.
$$
On the other hand, one can compute that
$$
\mathrm{trace}_{g}(\Psi_{B}^{*}g_{B})=-(f\circ \Psi_{B})^2\|\nabla u\|^2+ \dfrac{1}{(f\circ \Psi_{B})^2}\|\nabla v\|^2= -2 g(\xi^{\top}, \eta^{\top}),
$$
where $\mathrm{trace}_{g}(\Psi_{B}^{*}g_{B})=\sum_{i=1}^n (\Psi_{B}^{*}g_{B})(E_{i}, E_{i})$ for a local orthonormal frame in $M$.

\begin{remark}
{\rm  Taking into account that
$
n=\mathrm{trace}_{g}(g)=\mathrm{trace}_{g}(\Psi_{B}^{*}g_{B})+(\lambda\circ \Psi_{B})^2 \mathrm{trace}_{g}(\Psi_{F}^{*}g_{F}),
$
we get  $g(\xi^{\top}, \eta^{\top})>-n/2$.
}
\end{remark}

\smallskip

From (\ref{17112022G}), the above computations give the following result.

\begin{lemma}\label{160523A}
Let $\Psi:M\rightarrow B\times_{\lambda} F$ be a spacelike immersion in a generalized Schwarzschild  spacetime. Then the following formula holds
$$
\mathrm{div}(\xi^{\top}) - n\,\widetilde{g}(\mathbf{H},\xi^{\bot})  = \Big(\dfrac{\xi\lambda}{\lambda}\circ \Psi_{B}\Big)\left[n+2g(\xi^{\top}, \eta^{\top})\right] - (ff'\circ \Psi_{B}) \| \xi^{\top}\|^{2}. 
$$
\end{lemma}

\begin{remark}
{\rm We know that $\xi^{\top}=0$ for a spacelike immersion which factors through an integral hypersurface  $\mathcal{L}$ of $D_{\xi}$. In this case, the above Lemma reduces to 
$$
\widetilde{g}(\mathbf{H},\xi)  = -\dfrac{\xi\lambda}{\lambda}\circ \Psi_{B}. 
$$
}
\end{remark}

The following  result
 provides an integral characterization for  compact spacelike immersions through integral hypersurfaces of $D_{\xi}$.

\begin{theorem}\label{070323C}
Assume $M$ is compact and $\Psi:M\rightarrow B\times_{\lambda} F$ is a spacelike immersion in a generalized Schwarzschild spacetime  with $f'> 0$ (resp. $f'<0$). Then 
\begin{equation}\label{160523C}
\int_{M} \Big[ n\,\widetilde{g}(\mathbf{H},\xi^{\bot}) +\Big(\dfrac{\xi\lambda}{\lambda}\circ \Psi_{B}\Big)\left[n+2g(\xi^{\top}, \eta^{\top})\right]\Big]\, d\mu_{g} \geq 0.\quad (\textrm{resp.} \leq 0).
\end{equation}
The equality holds if and only if $M$ factors through an integral hypersurface  $\mathcal{L}$ of $D_{\xi}.$
\end{theorem}
\begin{proof}
Suppose the case $f'> 0$.   The inequality (\ref{160523C}) is a direct consequence of Lemma \ref{160523A} and the classical divergence theorem.
Furthermore,   (\ref{160523C}) becomes an equality if and only if
$$
\int_{M}(ff'\circ \Psi_{B})\| \xi^{\top}\|^{2}\, d\mu_{g}=0.
$$
Since $f'>0$ and the immersion is spacelike, we get  $\xi^{\top}=0$ and this fact ends the proof.
The proof for $f'<0$ works in a similar way.

\end{proof}

\begin{remark}\label{120823A}
{\rm  We have for the family of functions given in (\ref{app}) that
$$
 (ff')(r)=\dfrac{-2\Lambda r^{2m}+ m(m^2-1)(\mathbf{M}r^{m-1}- q^2)}{m(m+1)r^{2m-1}}.
$$
Therefore the assumption $f'>0$ is satisfied when $\Lambda\leq 0$ and $\mathbf{M}v^{m-1} >q^2$. In particular, it holds for the exterior Schwarzschild  spacetime.  

}

\end{remark}

\bigskip

We just sketch the proof of Lemma \ref{160523A} for the lightlike vector field $\eta$.
The condition $\widetilde{g}(\xi, \eta)=-1$ implies $\nabla^{B}\eta= -\alpha \otimes \eta$ (see Remark \ref{300523A}) and then
$$
 \widetilde{\nabla}_{ V}(\eta\mid_{\Psi})=-\alpha(V^{B})\eta\mid_{\Psi}+\Big(\dfrac{\eta\lambda}{\lambda}\circ \Psi_{B}\Big) V^F. 
$$
From the Gauss and Weingarten formulas we have
\begin{equation}\label{240523A}
\nabla_V\eta^{\top}+\mathrm{II}(V,\eta^{\top})- A_{\eta^{\bot}}V+\nabla^{\perp}_V\,\eta^{\bot}=-\alpha(V^{B})\eta\mid_{\Psi}+\Big(\dfrac{\eta\lambda}{\lambda}\circ \Psi_{B}\Big) V^F,
\end{equation}and taking tangent parts, we get 
\begin{equation}\label{210523C}
\nabla_V\eta^{\top}-A_{\eta^{\bot}}V=-\alpha(V^{B})\eta^{\top}+\Big(\dfrac{\eta\lambda}{\lambda}\circ \Psi_{B}\Big) (V^F)^{\top}.
\end{equation}
Hence we have
\begin{equation*}
\begin{split}
\mathrm{div}(\eta^{\top}) - n\,\widetilde{g}(\mathbf{H},\eta^{\bot}) & = \Big(\dfrac{\eta\lambda}{\lambda}\circ \Psi_{B}\Big)\left[n+(f\circ \Psi_{B})^2\|\nabla u\|^2 -\dfrac{1}{(f\circ \Psi_{B})^2}\|\nabla v\|^2\right] \\
 & -(\Psi_{B}^{*}\alpha) (\eta^{\top}). 
\end{split} 
\end{equation*}
A straightforward computation from (\ref{15112022C}) shows that
$$(\Psi^{*}\alpha)(\eta^{\top})=\Big(\frac{ff'}{2}\circ \Psi_{B}\Big)\mathrm{trace}_{g}(\Psi_{B}^{*}g_{B}).$$

\noindent Therefore,   the corresponding version of  Lemma \ref{160523A} to the vector field $\eta$ reads as follows.

\begin{lemma}
Let $\Psi:M\rightarrow B\times_{\lambda} F$ be a spacelike immersion in a generalized Schwarzschild spacetime. Then the following formula holds

\begin{equation}\label{070323B}
\mathrm{div}(\eta^{\top}) - n\,\widetilde{g}(\mathbf{H},\eta^{\bot}) = \Big(\dfrac{\eta\lambda}{\lambda}\circ \Psi_{B}\Big)\left[n+2 g(\xi^{\top}, \eta^{\top})\right]+
(ff'\circ \Psi_{B})g( \xi^{\top}, \eta^{\top}).
\end{equation}
For a submanifold $M$ which factors through an integral hypersurface $\mathcal{N}$ of $D_{\eta}$ we have 
$$
\widetilde{g}(\mathbf{H},\eta)  = -\dfrac{\eta\lambda}{\lambda}\circ \Psi_{B}. 
$$
\end{lemma}

\begin{remark}
{\rm  From (\ref{070323B}) and for $M$ compact, we get
$$
\int_{M} \Big[ n\,\widetilde{g}(\mathbf{H},\eta^{\bot}) +\Big(\dfrac{\eta\lambda}{\lambda}\circ \Psi_{B}\Big)\left[n+ 2 g(\xi^{\top}, \eta^{\top})\right]\Big]\, d\mu_{g} =
-\int_{M}(ff'\circ \Psi_{B})g( \xi^{\top}, \eta^{\top})\, d \mu_{g}.
$$
The right-hand side of this integral formula has no prescribed sign although we impose $f'$ to be  signed.   Hence, a similar result to Theorem  \ref{070323C}  does not hold for the vector field $\eta$.
}
\end{remark}

\section{Immersions  through lightlike integral hypersurfaces}

In this section,  we will focus on spacelike immersions  $\Psi:M\rightarrow B\times_{\lambda} F$ through lightlike integral hypersurfaces of the distributions $D_{\xi}$ or $D_{\eta}.$ 
We write $\mathcal{L}$ (resp.  $\mathcal{N}$) for a  general integral hypersurface of  $D_{\xi}$ (resp.  $D_{\eta}$).
The following results specialize several formulas of the above sections for theses cases. For example, the formula (\ref{21112022A}) using Lemma \ref{020323A}. 

\begin{lemma}\label{210523D}
Let $\Psi:M\rightarrow B\times_{\lambda} F$ be a spacelike immersion   through an integral hypersurface $\mathcal{L}$ of $D_{\xi}$. The following formulas hold
\begin{enumerate}
\item $\alpha(V^B)=0$ for every $V\in \mathfrak{X}(M).$

\item $\eta^{\top}=- \nabla v$.

\item $V^{B}= g(V, \nabla v)\xi$.   In particular, we get  $(V^B)^{\top}=0$ and $(V^{F})^{\top}=V.$

\item When $M$ has codimension two,   the normal tangent bundle $TM^{\perp}$ is spanned by the normal lightlike vector fields $\xi$ and 
$$
\ell^{\xi}=-\frac{\| \nabla v\|^{2}}{2}\xi +\eta^{\perp}.
$$
We also have  $\widetilde{g}(\xi, \ell^{\xi})=-1$.

\end{enumerate}

\end{lemma}

\begin{lemma}\label{080623A}
Let $\Psi:M\rightarrow B\times_{\lambda} F$ be a spacelike immersion   through an integral hypersurface $\mathcal{N}$ of $D_{\eta}$. The following formulas hold
\begin{enumerate}
\item $\alpha(V^B)=-\left(\dfrac{2f'}{f}\circ \Psi_{B}\right) g(V, \nabla v)$ for every $V\in \mathfrak{X}(M).$

\item $\xi^{\top}=\left(\dfrac{2}{f^2}\circ \Psi_{B}\right) \nabla v$.

\item $V^{B}=-\left(\dfrac{2}{f^2}\circ \Psi_{B}\right) g(V,\nabla v)\eta$. In particular, we get $(V^B)^{\top}=0$ and  $(V^{F})^{\top}=V.$

\item When $M$ has codimension two,  the normal tangent bundle $TM^{\perp}$ is spanned by the normal null vector fields $\eta$ and 
$$
\ell^{\eta}=\xi^{\perp}-\frac{2\| \nabla v\|^{2}}{(f\circ \Psi_{B})^4}\eta.
$$
We also have  $\widetilde{g}(\eta, \ell^{\eta})=-1$.
\end{enumerate}
\end{lemma}

\begin{remark}\label{110923B}
 {\rm    In these cases,   the projection $\Psi_{F}\colon M \to F$ is also an immersion.
Indeed, from Lemma \ref{210523D}, the equality $T_{x}\Psi_{F}\cdot v=0$ for $x\in M$ with $v\in T_{x}M$ and $\Psi(M)\subset \mathcal{L}$ give $g(v,v)=0$ and then $v=0$ . The same argument works for $\mathcal{N}$.
For spacelike immersions  $\Psi:M\rightarrow B\times_{\lambda} F$ through these lightlike integral hypersurfaces the induced metric is
$
g=(\lambda\circ \Psi_{B})^{2}\Psi_{F}^{*}(g_{F}).
$
}
\end{remark}

\begin{lemma}\label{16032023A}
Let $\Psi:M\rightarrow B\times_{\lambda} F$ be a spacelike immersion through an integral hypersurface  $\mathcal{L}$ of $D_{\xi}$.   For every  $V\in \mathfrak{X}(M)$  we have
$$
A_{\xi}V=-\Big(\frac{\xi \lambda}{\lambda}\circ \Psi_{B}\Big)V\,\, \textrm{ and }\,\,A_{\eta^{\perp}}V=-\Big(\frac{\eta \lambda}{\lambda}\circ \Psi_{B}\Big)V- \,  \nabla_{V }\nabla v.
$$
 In particular, we get
$
\widetilde{g}(\mathbf{H}, \eta^{\perp})=-\frac{\eta \lambda}{\lambda}\circ \Psi_{B} - \frac{1}{n} \Delta v.
$
\end{lemma}
\begin{proof}
Under our assumptions $\xi^{\top}=0$ and by means of Lemma \ref{210523D},  the   equation (\ref{270223B}) reduces to the announced  formula for $A_{\xi}$.
From Lemma \ref{210523D} we have $\eta^{\top}=-\nabla v$.    Hence,  formula (\ref{210523C})  ends the proof.
\end{proof}

In a similar way we have.
\begin{lemma}\label{16032023B}
Let $\Psi:M\rightarrow B\times_{\lambda} F$ be a spacelike immersion  through  an integral hypersurface $\mathcal{N}$ of $D_{\eta}$.  For every  $V\in \mathfrak{X}(M)$  we have
$$
A_{\eta}V=-\Big(\dfrac{\eta\lambda}{\lambda}\circ \Psi_{B}\Big)V
\quad \textrm{ and }\quad A_{\xi^{\perp}}V=-\Big(\frac{\xi \lambda}{\lambda}\circ \Psi_{B}\Big)V + \frac{2}{(f \circ \Psi_{B})^2} \nabla_{V }\nabla v.
$$
In particular, we get
$
\widetilde{g}(\mathbf{H}, \xi^{\bot})=-\frac{\xi \lambda}{\lambda}\circ \Psi_{B}+ \frac{2}{n (f\circ\Psi_{B})^2} \Delta v.
$
\end{lemma}

In order to avoid ambiguities, we add the following terminology. A function $f$ with values in $\R$ is said to be signed when $f\geq 0$ or $f \leq 0$
on whole domain. 
The assumptions on the functions $\eta \lambda $   and  $\xi\lambda$ in the statements of the following results are only required  along the immersions $\Psi$.

\begin{corollary}\label{020623A}
Assume $\eta \lambda $   $($resp.  $\xi\lambda)$ signed and
let $\Psi:M\rightarrow B\times_{\lambda} F$ be a compact spacelike immersion through an integral hypersurface  $\mathcal{L}$ (resp. $\mathcal{N}$) of the distribution $D_{\xi}$ $($resp. $D_{\eta})$ with $\mathbf{H}=0$. Then $M$ factors through a slice with $u=t_{0}$ and $v=r_{0}$ such  that $\nabla^{B}\lambda (t_{0}, r_{0})=0$  and  $M$ is minimal in $F$. 
\end{corollary}
\begin{proof} We give the proof only for the case of $D_{\xi}.$
From Lemma \ref{16032023A}, we have 
$$
\frac{\eta \lambda}{\lambda}\circ \Psi_{B} + \frac{1}{n} \Delta v=0.
$$
The assumption the sign of $\eta \lambda$  implies that $\Delta v$ is also signed.  
The compactness of $M$ gives that $v$ is a constant function $r_{0}$ and Lemma \ref{020323A} implies that $u$ is also  a constant function $t_{0}$. 
Now, consider the string of smooth maps
$$
M\stackrel{\Psi_{F}}{\longrightarrow }  F\hookrightarrow  B\times_{\lambda} F,
$$
where the second map is the slice at level $(t_{0}, r_{0}).$
From \cite[Chap. 3]{Chen}, we know that for an orthonormal local frame $(E_{1}, \cdots , E_{n})$ of the induced metric $g$, we have
$$
\mathbf{H}= \mathbf{H}'+\frac{1}{n}\sum_{i=1}^{n} \overline{\mathrm{II}}(T\Psi_F\cdot E_{i},T\Psi_F\cdot E_{i}), 
$$
 where $\mathbf{H}'$ denotes the mean curvature vector field of $\Psi_{F}$ and $\overline{\mathrm{II}}$ is the second fundamental form of the slice at $(t_{0}, r_{0}).$
Therefore, as a consequence of (\ref{110923A}) and Remark \ref{110923B}, we get that $$
\mathbf{H}= \mathbf{H}'-\dfrac{\nabla^{B} \lambda}{\lambda}(t_{0}, r_{0}).
$$
Hence 
$\mathbf{H}=0$ provides that $\mathbf{H}'=0$ and  $\nabla^{B} \lambda(t_{0}, r_{0})=0.$ 
\end{proof}



As a direct consequence of  Lemmas \ref{16032023A} and  \ref{16032023B}, for a spacelike submanifold through an integral hypersurface of $D_{\xi}$ (resp. $D_{\eta})$, the  normal vector field $\eta^{\perp}$ (resp. $\xi^{\perp}$) is  umbilic if and only if there is $h\in\mathcal{C}^{\infty}(M)$ such that
\begin{equation}\label{22032023A}
\nabla_V\nabla v=hV,
\end{equation}
for every $V\in\mathfrak{X}(M)$.  
Then, in both situations, we have $\mathcal{L}_{\nabla v}g=2hg$ and therefore $\nabla v$ is a conformal vector field in $(M,g)$.

\begin{theorem}\label{20032023A}
Let $\Psi:M\rightarrow B\times_{\lambda} F$ be a compact spacelike immersion  through an integral hypersurface $\mathcal{L}$ (resp.  $\mathcal{N}$) of $D_{\xi}$ $($resp. $D_{\eta})$ with $\mathrm{Ric}^g(\nabla v,\nabla v)\leq 0$. Then $\eta^{\perp}$ $($resp. $\xi^{\perp})$ is an umbilic direction  if and only if $M$ factors through a slice.
\end{theorem}
\begin{proof}
Since $\nabla v$ is a conformal vector field and $\textrm{Ric}^g(\nabla v,\nabla v)\leq 0$ it follows that $\nabla v$ is a Killing vector field (see \cite[Chap. 5]{Poor}). This necessarily implies that $h=0$ in (\ref{22032023A}) and therefore $\Delta v=0$. Taking into account that $M$ is assumed to be compact, we get that $v$ is a constant function. Furthermore, from Lemma \ref{020323A} the function $u$ must also be a constant. The converse is obvious.

\end{proof}

Equation (\ref{22032023A})  implies that $\nabla v$ is a conformal gradient vector field on the Riemannian manifold $(M,g)$. From a classic result by Obata \cite{Ob}, the existence of such vector fields on compact Riemannian manifolds has been address in \cite[Theor. 1]{DAL}. As a direct consequence of this result, we have.

\begin{theorem}\label{260723A}
Let $\Psi:M\rightarrow B\times_{\lambda} F$ be a  compact spacelike immersion   through an integral hypersurface $\mathcal{L}$ (resp.  $\mathcal{N}$) of $D_{\xi}$ $($resp. $D_{\eta})$ with  $\eta^{\perp}$ $($resp. $\xi^{\perp})$  an umbilical direction.
Assume the Ricci tensor of the induced metric $g$ satisfies
$$
0< \mathrm{Ric}^{g}\leq (n-1)\left( 2- \frac{nc}{\lambda_{1}}\right)c,
$$
for a constant $c$ where $\lambda_{1}$ is the first non-trivial eigenvalue of the Laplace operator of the metric $g$.  If $\nabla v$ is a nonzero vector field, then $(M,g)$ is isometric to the sphere $\S^{n}(c)$ of constant sectional curvature $c$.

\end{theorem}

\begin{remark}
{\rm A careful reading of \cite[Theor. 1]{DAL} shows that we only need the above condition on the Ricci tensor for the vector field $\nabla (\frac{\Delta v}{n}+c v)$.

}
\end{remark}

\begin{remark}
{\rm From \cite[Chap. 1]{Y}, we know that for every conformal vector field $V$ in an $n$-dimensional Riemannian manifold, the following formula holds
$$
V(S^g)=-\frac{2(n-1)}{n}\Delta(\mathrm{div} V)- \frac{2}{n}\mathrm{div}V\cdot S^g,
$$
where $S^g$ is the scalar curvature of $g$. Under the assumptions of Theorem \ref{260723A}, the vector field $\nabla v$
is  conformal and the manifold $(M,g)$ is isometric to the sphere $\S^{n}(c)$. Therefore, the above formula reduces  to
$$
0=\Delta(\Delta v)+n c\Delta v.
$$
This implies that $\Delta v+ nc\, v=k \in \R$. Hence for $w:=v-\frac{k}{nc}$, we have that $\Delta w+ nc w=0$ and
 either $v=\frac{k}{nc}$ or $w$ is an eigenfunction of the Laplace operator of the sphere $\S^n(c)$ corresponding with $n c$. This is the well-known value of the first non-trivial eigenvalue $\lambda_1$ of the Laplace operator of the sphere $\S^n(c)$, \cite[Chap. 2]{Cha}. The space of homogeneous  harmonic polynomial of $\R^{n+1}$  of degree $1$ restricted to $\S^{n}(c)$  constitutes the eigenspace corresponding to $\lambda_{1}$. In other words, there is $a\in \R^{n+1}$ such that
$
v(x)=\langle x, a\rangle+ \frac{k}{nc}, 
$
where $\langle \,,\,\rangle$ is the usual Euclidean inner product.
}
\end{remark}

\begin{proposition}\label{28032023A}
Let $\Psi:M\rightarrow B\times_{\lambda} F$ be a spacelike immersion  through an integral hypersurface $\mathcal{L}$  of $D_{\xi}$.  For every $V\in\mathfrak{X}(M)$  we have 
$$
\nabla_V^{\perp}\xi=-\left(\dfrac{\xi\lambda}{\lambda}\circ\Psi_B\right)g(\nabla v,V)\xi\,\, \textrm{ and }\,\,
\nabla_V^{\perp}\eta^{\perp}=-\left(\dfrac{\eta\lambda}{\lambda}\circ\Psi_B\right)g(\nabla v,V)\xi+\mathrm{II}(\nabla v,V).
$$
\end{proposition}
\begin{proof} The first assertion is a direct computation from equation (\ref{17112022B}) from Lemmas \ref{210523D} and \ref{16032023A} taking into account  that $V^{F}-V=-V^{B}$.
On the other hand,   Lemmas \ref{210523D} and \ref{16032023A} reduce equation (\ref{240523A})  to
$$
-\mathrm{II}(V,\nabla v)+\Big( \frac{\eta \lambda}{\lambda}\circ \Psi_{B}\Big) V+\nabla^{\perp}_V\,\eta^{\bot}=\Big(\dfrac{\eta\lambda}{\lambda}\circ \Psi_{B}\Big) V^F.
$$
and  again from $V^{F}-V=-V^{B}$, the above formula ends the proof.

\end{proof}

In a similar way we obtain the following lemma.
\begin{proposition}\label{28032023B}
Let $\Psi:M\rightarrow B\times_{\lambda} F$ be a spacelike immersion   through an integral hypersurface  $\mathcal{N}$ of $D_{\eta}$.  For every $V\in\mathfrak{X}(M)$  we have
$$
\nabla_V^{\perp}\eta=\dfrac{2}{(f\circ\Psi_B)^2}\left(ff'\circ\Psi_B+\dfrac{\eta\lambda}{\lambda}\circ\Psi_B\right)g(\nabla v,V)\eta
$$
and
$$
\nabla_V^{\perp}\xi^{\perp}=\dfrac{2}{(f\circ\Psi_B)^2}\left[g(\nabla v,V)\left(-\left(ff'\circ\Psi_B\right)\xi^{\perp}+\left(\dfrac{\xi\lambda}{\lambda}\circ\Psi_B\right)\eta\right)-\mathrm{II}(\nabla v,V)\right].
$$
\end{proposition}

As a direct consequence of Propositions  \ref{28032023A} and \ref{28032023B} we have.

\begin{theorem}\label{250523A}
Let $\Psi:M\rightarrow B\times_{\lambda} F$ be a spacelike immersion  through an integral hypersurface $\mathcal{L}$ (resp.  $\mathcal{N}$) of $D_{\xi}$  (resp. $D_{\eta}$).  Assume $\xi\lambda\neq 0$ $($resp. $ff'+\frac{\eta\lambda}{\lambda}\neq 0$), then the following assertions are equivalent
\begin{enumerate} 
\item $\nabla_V^{\perp}\xi=0$ $($resp. $\nabla_V^{\perp}\eta=0)$ for every $V\in\mathfrak{X}(M)$. 
\item  $M$ factors through a slice.
\end{enumerate}
\end{theorem}

\begin{remark}
\rm{ In order to study the applicability of this result to the case $B\times_{r}\S^{m}$ where $f^{2}(r)$ is given in (\ref{app}),  recall $\xi \lambda=1$ and a direct computation shows that $ff'+\frac{\eta \lambda}{\lambda}\neq 0$ if and only if 
$$
-2\Lambda r^{m+1}+m(m+1)(2m\mathbf{M}-r^{m-1}-(2m-1)q^2r^{1-m})\neq 0.
$$
For the exterior Schwarzschild spacetime with mass $\mathbf{M}$, we have  
$$
ff'+\frac{\eta\lambda}{\lambda}=\frac{2m\mathbf{M}-r^{m-1}}{2r^m}.
$$
Therefore, the assumption in Theorem \ref{250523A} is always satisfied for the distribution $D_{\xi}$ but the condition for $D_{\eta}$ holds if and only if the value 
$ 2m\mathbf{M}$ is not achieved for the function $v^{m-1}\in C^{\infty}(M)$.
}
\end{remark}

\section{Codimension two spacelike immersions }

From now on, we assume $n=m$, that is, the spacelike immersion $\Psi:M\rightarrow B\times_{\lambda} F$ has codimension two. 
We begin this section with a topological result on such spacelike immersions.

\begin{proposition}\label{020623B}
Let $\Psi:M\rightarrow B\times_{\lambda} F$ be a codimension two compact spacelike  immersion  through an integral hypersurface $\mathcal{L}$ (resp.  $\mathcal{N}$)  of $D_{\xi}$  $($resp. $D_{\eta})$. 
Then the map $\Psi_{F}\colon M \to F$ is a covering map. In particular,  $F$ is also compact and when $F$ is  simply-connected,  $\Psi_{F}$ is a diffeomorphism. 
\end{proposition}
\begin{proof} We give the proof only for the case of the distribution $D_{\xi}$.  We claim that  the map $\Psi_{F}\colon M \to F$ is a local diffeomorphism.
Indeed,  we know that $\Psi_{F}\colon M \to F$ is an immersion between manifolds of the same dimension. The compactness of $M$ and the connectedness of $F$ imply that $\Psi_F$ is a covering map
(see \cite[Proposition  5.6.1]{Carmo1} for details).
\end{proof}
\begin{remark}
 {\rm  Under the assumptions of Theorem \ref{260723A}, if $F$ is simply-connected necessarily must be a topological sphere. }
\end{remark}
\begin{remark}\label{140823A}
{\rm  The map $\Psi_{F}\colon M \to F$ is not a Riemannian covering, in general.
In fact,   as was mentioned  for  $\Psi=(\Psi_{B}, \Psi_{F})$, the induced metric on $M$ is given by
$
g=(\lambda \circ \Psi_{B})^{2}\Psi_{F}^{*}(g_{F}).
$
Taking into account the well-known relation  between the scalar curvatures of two conformally related metrics, see for instance \cite[Chap. 1]{Besse}, we have that
$$
S^{\Psi_{F}^{*}(g_{F})}=(\lambda \circ \Psi_{B})^2 \left(S^{g}+2(m-1)\Delta\log(\lambda\circ\Psi_B)-(m-2)(m-1)\|\nabla\log(\lambda\circ\psi_B)\|^2\right),
$$
where $S^{\Psi_{F}^{*}(g_{F})}$ and $S^g$ are the scalar curvatures of $\Psi_{F}^{*}(g_{F})$ and $g$, respectively. 
For $\lambda(t,r)=r$, we have $\lambda \circ \Psi_{B}=v$  and from a straightforward computation  the  above formula reads as follows
\begin{equation}\label{030623A}
S^{\Psi_{F}^{*}(g_{F})}=v^2 \left(S^{g}+\dfrac{2(m-1)}{v}\Delta v-\dfrac{m(m-1)}{v^2}\|\nabla v  \|^2\right).
\end{equation}

}
\end{remark}

\begin{proposition}\label{080723B}
Assume $\lambda(t,r)=r$ and let $\Psi:M\rightarrow B\times_{\lambda} F$ be a compact codimension two spacelike immersion  through an integral hypersurface of $D_{\xi}$ or $D_{\eta}$.   Then
$$
\int_{M}\left(S^{\Psi_{F}^{*}(g_{F})}- v^2S^g\right) d\mu_{g} \leq 0
$$
and the equality holds if and only if $M$ factors through a slice.
\end{proposition}
\begin{proof}
From (\ref{030623A}), a direct computation gives that
$$
\int_{M}\left(S^{\Psi_{F}^{*}(g_{F})}- v^2S^g\right) d\mu_{g}=(m-1)\int_{M}(2v\Delta v-m \| \nabla v\|^{2})d\mu_{g}.
$$
Taking into account that $\Delta v^2= 2v \Delta v+ 2 \| \nabla v\|^{2}$, the above formula reduces to
$$
\int_{M}\left(S^{\Psi_{F}^{*}(g_{F})}- v^2S^g\right) d\mu_{g}=-(m-1)(m+2)\int_{M} \| \nabla v\|^{2}d\mu_{g}.
$$
Then, Lemma \ref{020323A} ends the proof.
\end{proof}

As a direct consequence of Corollary \ref{020623A} and Proposition \ref{020623B} we have.

\begin{corollary}
Assume $\eta  \lambda$ $($resp.  $\xi \lambda )$ signed. Then every  codimension two compact spacelike immersion  through an integral hypersurface  $\mathcal{L}$ (resp.  $\mathcal{N}$) of $D_{\xi}$  $($resp. $D_{\eta})$  with $\mathbf{H}=0$  factors through a slice  and $\Psi_{F}\colon (M, g)\to (F, \lambda(t_{0}, r_{0})^{2}g_{F})$ is a Riemannian covering space.
\end{corollary}

\begin{proposition}\label{010623B}
Let $\Psi:M\rightarrow B\times_{\lambda} F$ be a codimension two spacelike immersion  through an integral hypersurface $\mathcal{L}$ (resp.  $\mathcal{N}$) of $D_{\xi}$  $($resp. $D_{\eta})$.  Then  the normal tangent bundle $TM^{\perp}$ is spanned by the vector fields $\xi^{\perp}$ and $\eta^{\perp}$.  When $M$ factors through an integral hypersurface $\mathcal{L}$ of $D_{\xi}$,
 the mean curvature vector field is
\begin{equation}\label{Hxi}
\mathbf{H}=\left[\dfrac{\eta\lambda}{\lambda}\circ \Psi_{B}-\Big(\dfrac{\xi\lambda}{\lambda}\circ \Psi_{B}\Big)\|\nabla v\|^2+\dfrac{1}{m}\Delta v\right]\xi+\Big(\dfrac{\xi\lambda}{\lambda}\circ \Psi_{B}\Big)\eta^{\perp}
\end{equation}
and when $M$ factors through an integral hypersurface $\mathcal{N}$ of $D_{\eta}$ we have
$$
\mathbf{H}=\Big(\dfrac{\eta\lambda}{\lambda}\circ \Psi_{B}\Big)\xi^{\bot}+\left[\dfrac{\xi\lambda}{\lambda}\circ \Psi_{B}-\Big(\dfrac{4\,\eta\lambda}{\lambda f^4}\circ \Psi_{B}\Big)\|\nabla v\|^2-\dfrac{2}{m(f\circ \Psi_{B})^2}\Delta v\right]\eta.
$$
\end{proposition}
\begin{proof} The assertion on the normal tangent bundle is a direct consequence of $\xi^{\bot}=\xi$ (resp.  $\eta^{\bot}=\eta$) and $\widetilde{g}(\xi, \eta)=-1.$
Hence,  there are smooth functions $a,b\in C^{\infty}(M)$ such  that
$$
\mathbf{H}=a\xi+b\eta^{\perp}
$$
where $b=-\widetilde{g}(\mathbf{H}, \xi)$ and  from Lemma \ref{210523D} we can compute that $a=\widetilde{g}(\mathbf{H}, \|\nabla v\|^2\xi-\eta^{\perp}).$
Now  formula (\ref{230321C})  and Lemma \ref{16032023A}  imply that
$$
\widetilde{g}(\mathbf{H}, \xi)=-\dfrac{\xi\lambda}{\lambda}\circ \Psi_{B}\quad \textrm{ and } \quad \widetilde{g}(\mathbf{H}, \eta^{\perp})=-\dfrac{\eta\lambda}{\lambda}\circ \Psi_{B}-\dfrac{1}{m}\Delta v.
$$
This completes the proof for the case of spacelike submanifolds through integral hypersurfaces of $D_{\xi}$.  Slight changes in the proof show the formula for the mean curvature vector field in the case $D_{\eta}.$

\end{proof}

Under the same assumptions of Proposition \ref{010623B}
and from Lemmas \ref{210523D} and \ref{080623A}, we have.

\begin{corollary}\label{280723D}
For a spacelike submanifold  $M$ which factors through an integral hypersurface $\mathcal{L}$ of $D_{\xi}$, we have
$$
\mathbf{H}=\left[\dfrac{\eta\lambda}{\lambda}\circ \Psi_{B}-\Big(\dfrac{\xi\lambda}{2\lambda}\circ \Psi_{B}\Big)\|\nabla v\|^2+\dfrac{1}{m}\Delta v\right]\xi+\Big(\dfrac{\xi\lambda}{\lambda}\circ \Psi_{B}\Big)\ell^{\xi}.
$$
In case that $M$ factors through an integral hypersurface  $\mathcal{N}$  of $D_{\eta}$, 
$$
\mathbf{H}=\Big(\dfrac{\eta\lambda}{\lambda}\circ \Psi_{B}\Big)\ell^{\eta}+\left[\dfrac{\xi\lambda}{\lambda}\circ \Psi_{B}-\Big(\dfrac{2\,\eta\lambda}{\lambda f^4}\circ \Psi_{B}\Big)\|\nabla v\|^2-\dfrac{2}{m(f\circ \Psi_{B})^2}\Delta v\right]\eta.
$$

\end{corollary}

\begin{remark}\label{140823B}
{\rm This Corollary extends the formulas for the mean curvature vector field of codimension two spacelike submanifolds through lightlike hyperplanes and cones in the Minkowski spacetime, \cite{ACR2} and \cite{PaRo}. For submanifolds through lightlike hyperplanes, we particularize Corollary \ref{280723D}
for $f(r)=1$ and $\lambda(t,r)=1$  (see Example \ref{280723A}), then
$
\mathbf{H}=\frac{1}{m}\Delta v\, \xi
$
for $M$  in the lightlike hyperplane  $\Pi_{\xi}:= \{x\in \L^{m+2}: \langle x, \xi\rangle=0\}$. 
Similarly, we obtain    
$
\mathbf{H}=-\frac{2}{m}\Delta v \, \eta
$
when $M$ factors through the lightlike hyperplane $\Pi_{\eta}$. It can be easily seen that our formulas for $\mathbf{H}$ coincides with the formula (8.2) given in \cite[Sect. 8]{ACR2}.

For submanifolds through lightlike cones we need to take $f(r)=1$ and $\lambda(t,r)=r$, then we compute the mean curvature as follows
$$
\mathbf{H}=\Big( \frac{-1-\|\nabla v\|^2}{2v}+ \dfrac{\Delta v}{m}\Big)\xi +  \frac{1}{v}\ell^{\xi},\quad\textrm{{$D_{\xi}$ case}}
$$
and
$$
\mathbf{H}=-  \frac{1}{2v}\ell^{\eta}+\Big( \dfrac{1+\|\nabla v\|^2}{v}-\frac{2\Delta v}{m}\Big)\eta, \quad\textrm{{$D_{\eta}$ case}}.
$$
Theses formulas agree with \cite{PaRo} and  \cite[Sect. 6]{ACR2}. In fact, let us denote by $\Lambda^+\subset\L^{m+2}$  the future lightlike cone with vertex at $0\in \L^{m+2}.$ We know from (\ref{16082023-1}) that  $\hat{\xi}:=r\xi$ satisfies  $T\varphi\cdot\hat{\xi}$ is the position vector field in $\Lambda^+$. Therefore, if we rescale $\hat{l}:=\frac{1}{r}l^{\xi}$, the first formula of $\mathbf{H}$ expressed in terms of $\hat{\xi}$ and $\hat{l}$ coincides with the formula given in  \cite{PaRo} and \cite[Sect. 6]{ACR2}. 
 }
\end{remark}

\begin{remark}\label{060823B}
{\rm Assume the warping function $\lambda$ depends only on the radial coordinate $r$.
According to (\ref{Hxi}), the existence of a codimension two  spacelike  immersion with $\mathbf{H}=0$ through an integral hypersurface $\mathcal{L}$  of $D_{\xi}$ 
implies
$
 \lambda_{r} \circ \Psi_{B}=0.
$
In particular, there are no such spacelike immersions in the exterior Schwarzschild spacetime with mass $\mathbf{M}$.
The same result remains true for  spacelike  immersions   through an integral hypersurface $\mathcal{N}$  of $D_{\eta}.$ 

}

\end{remark}

\begin{remark}
{\rm Let $\Psi:M\rightarrow B\times_{\lambda} F$ be a codimension two spacelike immersion  through an integral hypersurface $\mathcal{L}$ of $D_{\xi}$ 
 with $\lambda(t,r)=r$.  Proposition \ref{28032023A} and Corollary \ref{280723D} give
that the normal lightlike vector field $\ell := v \,\xi$ satisfies
\begin{equation}\label{110823A}
\widetilde{g}(\ell, \mathbf{H})=-1\quad  \textrm{and} \quad \nabla^{\perp}\ell=0.
\end{equation}
In the case that the spacelike immersion  lies in an integral hypersurface $\mathcal{N}$ of $D_{\eta}$, the lightlike normal vector field $\ell:=\frac{-2v}{(f\circ \Psi_{B})^2}\eta$ satisfies 
$
\widetilde{g}(\ell, \mathbf{H})=-1$ and  $\nabla^{\perp}\ell=0.
$
Assuming that $M$ is compact and $F$ is the round sphere, the existence of a normal lightlike vector field $\ell $ such that (\ref{110823A}) holds implies that $M$ lies in an integral  lightlike hypersurface of $D_{\xi}$ or $D_{\eta}$, \cite{WWZ}.
The authors of \cite{WWZ} call these hypersurfaces as null hypersurfaces of symmetry.  In other words, the null hypersurfaces generated by the round sphere.

}

\end{remark}

From formulas $(2)$ in Lemma \ref{210523D} and Lemma \ref{080623A} and  Proposition  \ref{010623B}, we get.

\begin{corollary}\label{16032023E}
Under the same assumptions of Proposition \ref{010623B},  for a spacelike immersion through an integral hypersurface $\mathcal{L}$ of $D_{\xi}$  we have that
$$
\|\mathbf{H}\|^{2}=-\frac{2\, \eta \lambda\,  \xi \lambda}{\lambda^{2}}\circ \Psi_{B}- \Big(\frac{2\, \xi \lambda}{m\, \lambda}\circ \Psi_{B}\Big)\Delta v+\Big(\frac{\xi \, \lambda}{\lambda}\circ \Psi_{B}\Big)^{2}\|\nabla v\|^{2}
$$
and for a spacelike immersion through an integral hypersurface $\mathcal{N}$ of $D_{\eta}$ 
$$
\|\mathbf{H}\|^{2}=-\frac{2\, \eta \lambda\,  \xi \lambda}{\lambda^{2}}\circ \Psi_{B}+ \Big(\frac{4\, \eta \lambda}{m\, \lambda f^{2}}\circ \Psi_{B}\Big)\Delta v+4\Big(\frac{\eta \, \lambda}{\lambda f^{2}}\circ \Psi_{B}\Big)^{2}\|\nabla v\|^{2}
$$
where  as usual $\|\mathbf{H}\|^{2}$ denotes $\widetilde{g}(\mathbf{H}, \mathbf{H})$.
\end{corollary}

\begin{corollary}
Assume $\xi  \lambda$ $($resp.  $\eta \lambda )$ signed. Then every  codimension two compact spacelike  immersion  through an integral hypersurface $\mathcal{L}$ (resp.  $\mathcal{N}$) of $D_{\xi}$ (resp.  $D_{\eta}$)  with $$
\|\mathbf{H}\|^{2}=-\frac{2\, \eta \lambda\,  \xi \lambda}{\lambda^{2}}\circ \Psi_{B}
$$  factors through a slice.  
\end{corollary}
\begin{proof}
We give the proof only for the case of $D_{\xi}$. As a consequence of Corollary \ref{16032023E},  we have
$$
\Big(\frac{2\, \xi \lambda}{m\lambda}\circ \Psi_{B}\Big)\Delta v=\Big(\frac{\xi \, \lambda}{\lambda}\circ \Psi_{B}\Big)^{2}\|\nabla v\|^{2}.
$$
Since $\xi\lambda\circ\Psi_B$ is signed, we deduce that $\Delta v$ is also signed. We can now proceed analogously to the proof of Theorem \ref{20032023A}.

\end{proof}

\begin{remark}\label{070823A}
{\rm 
Corollary \ref{16032023E} provides  a formula which relates the mean curvature vector field and the scalar curvatures $S^{\Psi_{F}^{*}(g_{F})}$ and $S^g$. In fact, for a codimension two spacelike immersion through an integral hypersurface of $D_{\xi}$ or $D_{\eta}$, one computes from  (\ref{030623A})  that
\begin{equation}\label{080823A}
\|\mathbf{H}\|^{2}=\frac{1}{v^2}\left( (f\circ \Psi_{B})^2- \frac{S^{\Psi_{F}^{*}(g_{F})}- v^2  S^{g}}{m(m-1)}\right).
\end{equation}
If we specialize this formula for the case $m=2$,  we get that
$$
\|\mathbf{H}\|^{2}=\frac{1}{v^2}\left( (f\circ \Psi_{B})^2- K^{\Psi_{F}^{*}(g_{F})}+ v^2  K^g\right). 
$$
Hence,  if we assume $M$  compact, the Gauss-Bonnet formula implies
$$
\int_{M^2}\|\mathbf{H}\|^{2} \, d\mu_{g}=\int_{M^2}\frac{(f\circ \Psi_{B})^2}{v^2}\, d\mu_{g},
$$
where $d\mu_{g}$ is the canonical measure associated to the metric $g$.

}
\end{remark}

 Let us recall the following terminology in General Relativity.    

\begin{definition}\label{110823C}
{\rm A codimension two spacelike submanifold $M$ in a timelike orientable Lorentzian manifold is said to be future (resp.  past) trapped when $\mathbf{H}$ is timelike and future pointing (resp. past pointing).  $M$ is called marginally (resp. weakly) trapped if $\mathbf{H}$ is lightlike (resp. causal) on $M$.  The notions of future and past for marginally and weakly are obviously adapted.}
\end{definition}

\begin{remark}\label{080723A}
{\rm 
Assume $\lambda(t,r)=r$. Under the hypotheses of Corollary \ref{16032023E},   we have $\|\mathbf{H}\|^{2}\leq 0$ for a spacelike immersion  through an integral hypersurface of $D_{\xi}$ or $D_{\eta}$ if and only if
$$
\frac{2}{n }\Delta v\geq  \frac{ (f\circ \Psi_{B})^2+\|\nabla v\|^{2}}{v}.
$$
In particular, there are no compact weakly trapped  submanifolds in this case.  
Taking into account that for this choice of $\lambda$,  the manifold $B\times_{\lambda} F $ is stationary,  this result is only a particular case of  \cite[Theor. 2]{MS}.  In fact,
recall that  \cite[Theor. 2]{MS} states that there is no compact weakly trapped  submanifold in a stationary spacetime.
 }

\end{remark}

From Remark  \ref{060823B}, there is no point on $M$ where $\mathbf{H}=0$ and then, from (\ref{030623A}), we have.
\begin{corollary}\label{140823C}
Assume $\lambda (t,r)=r$ and let $\Psi:M\rightarrow B\times_{\lambda} F$ be a codimension two spacelike immersion  through an integral hypersurface of $D_{\xi}$ or $D_{\eta}$. Then the following assertions are equivalent.
\begin{enumerate}
\item $M$ is marginally trapped.
\item The function $v$ satisfies the equation
$$
2v \Delta v- m\Big[   (f\circ \Psi_{B})^2+\|\nabla v\|^{2}\Big]=0.
$$
\item The scalar curvature of $M$ satisfies
$$S^{\Psi_{F}^{*}(g_{F})}=v^2   S^{g}+ m(m-1)(f\circ \Psi_{B})^2.$$

\end{enumerate}

\end{corollary}

\smallskip

\begin{definition}
{\it  An immersion $\Psi:F\rightarrow B\times_{\lambda} F$ is said to be an (entire) spacelike graph on $F$ when
$$
\Psi(x)=(\Psi_B(x),x)
$$
and the induced metric $\Psi^{*}(\widetilde{g})$ is Riemannian.}
\end{definition}
\noindent
Recall that if a spacelike graph on $F$ factors through an integral hypersurface of $D_{\xi}$ or $D_{\eta}$, the induced metric is $g=(\lambda\circ \Psi_{B})^{2}g_{F}$.    Assume $\lambda(t,r)=r$.
Taking into account that $\Psi_{F}= \mathrm{Id}_{F}$ for spacelike graphs,  formula (\ref{080823A}) implies  that  for every graph  factoring through an integral hypersurface of $D_{\xi}$ or $D_{\eta}$,  the mean curvature vector field $\mathbf{H}$ and the scalar curvatures  $S^{g_{F}}$ and $S^g$ are related by
\begin{equation}\label{04092023}
 \|\mathbf{H}\|^{2}=\frac{1}{v^2}\left( (f\circ \Psi_{B})^2- \frac{S^{g_{F}}- v^2  S^{g}}{m(m-1)}\right).   
\end{equation}

\begin{remark}
{\rm In the particular case of the exterior Schwarzschild spacetime with mass $\mathbf{M}$,  the above formula reduces to 
$$
\|\mathbf{H}\|^{2}= \frac{S^g}{m(m-1)}- \frac{2\mathbf{M}}{v^{m+1}}.
$$
Then, a spacelike graph factoring through a lightlike hypersurface of $D_{\xi}$ or $D_{\eta}$ is marginally trapped if and only if $S^{g}=\frac{2\mathbf{M}m(m-1)}{v^{m+1}}.$ Also, taking into account the description of the Minkowski spacetime in Example \ref{280723A},  a spacelike graph factoring through a lightlike cone in the Minkowski spacetime is marginally trapped if and only if $S^{g}=0.$ }
\end{remark}

\begin{theorem}
Assume $\lambda$ depends only on the radial coordinate $r$ and $F$ is a non-compact  parabolic Riemannian manifold.
Let $\Psi:F\rightarrow B\times_{\lambda} F$ be a  spacelike  graph through an integral hypersurface $\mathcal{L}$ (resp.  $\mathcal{N}$) of $D_{\xi}$ (resp.  $D_{\eta}$) with $\mathbf{H}=0$. Then, $\Psi(F)$ is a totally geodesic slice.
\end{theorem}
\begin{proof} We give the proof only for $\mathcal{L}$ since the case of $\mathcal{N}$ is similar.
From formula (\ref{Hxi}), we get that $ \lambda\circ\Psi_B$ is a constant $k\in \R_{+}$ and  $\Delta v=0$.   Then (\ref{04092023}) implies that
$$
\dfrac{(k^2- v^2)S^{g}}{m(m-1)}= (f \circ \Psi_{B})^2.
$$
There are two possibilities. The first one is $S^{g}>0$ and $v^2< k^2$, then the parabolicity of $F$ gives that $v$ is constant. The other possibility is $S^g <0$ and $k^{2}< v^2$. Therefore, $-v<-k$ or $v<-k$ with $\Delta v=0$, again the parabolicity of $F$ shows that $v$ is constant. Now, formula (\ref{110923A}) ends the proof. 
\end{proof}

\section{Parallel mean curvature}

\begin{lemma}\label{040623A}
Let $\Psi:M\rightarrow B\times_{\lambda} F$ be a codimension two spacelike immersion  through an integral hypersurface $\mathcal{L}$ of $D_{\xi}$. 
 For every $V\in\mathfrak{X}(M)$, we have 
$$
\widetilde{g}(\nabla^{\perp}_V\mathbf{H},\xi)=-g(\nabla v,V)\left(\dfrac{\xi(\xi \lambda)}{\lambda}\circ \Psi_{B}\right).
$$
For the case of an integral hypersurface  $\mathcal{N}$ of $D_{\eta}$, we have
$$
\widetilde{g}(\nabla^{\perp}_V\mathbf{H},\eta)=2g(\nabla v,V)\left( \frac{\eta (\eta \lambda)+f' f \, \eta \lambda}{\lambda f^2}\circ \Psi_{B}\right).
$$
When $\lambda(t,r)=r$, the above formulas reduce to 
$
\widetilde{g}(\nabla^{\perp}_V\mathbf{H},\xi)=0$ 
and 
$\widetilde{g}(\nabla^{\perp}_V\mathbf{H},\eta)=0,
$ respectively.

\end{lemma}

\begin{proof}
From Propositions  \ref{28032023A}  and  \ref{010623B}, we derive  
\begin{equation}\label{130623A}
\widetilde{g}(\nabla^{\perp}_V\mathbf{H},\xi)=-V\left(\dfrac{\xi\lambda}{\lambda}\circ\Psi_B\right)-\left(\dfrac{\xi\lambda}{\lambda}\circ\Psi_B\right)^2g(\nabla v,V).
\end{equation}
A direct computation from Lemma \ref{210523D} shows that
$$
V\left(\dfrac{\xi\lambda}{\lambda}\circ\Psi_B\right)= V^{B}\left(\dfrac{\xi\lambda}{\lambda}\right)=g(\nabla v, V)\Big(\xi\left(\dfrac{\xi \lambda} {\lambda}\right)\circ \Psi_B\Big).
$$
Substituting this formula in (\ref{130623A}), we get
$$
\widetilde{g}(\nabla^{\perp}_V\mathbf{H},\xi)=-V\left(\dfrac{\xi\lambda}{\lambda}\circ\Psi_B\right)-\left(\dfrac{\xi\lambda}{\lambda}\circ\Psi_B\right)^2g(\nabla v,V)=-g(\nabla v,V)\Big(\dfrac{\xi(\xi \lambda)}{\lambda}\circ \Psi_B\Big).
$$
In a similar way,  from Propositions  \ref{28032023B} and  \ref{010623B}, we have
$$
\widetilde{g}(\nabla^{\perp}_V\mathbf{H},\eta)=-V\left(\dfrac{\eta\lambda}{\lambda}\circ\Psi_B\right)+\dfrac{2}{(f\circ\Psi_B)^2}\left(\dfrac{\eta\lambda}{\lambda}\circ\Psi_B\right)\left(ff'\circ\Psi_B+\dfrac{\eta\lambda}{\lambda}\circ\Psi_B\right)g(\nabla v,V)
$$
and Lemma \ref{080623A} ends the proof.
\end{proof}

\begin{theorem}\label{130623B}
Assume the warping function satisfies $\xi (\xi \lambda)\neq 0$ at every point and let $\Psi:M\rightarrow B\times_{\lambda} F$ be a codimension two spacelike immersion  through an integral hypersurface $\mathcal{L}$ of $D_{\xi}$.  Then the following assertions are equivalent
\begin{enumerate}
\item $\widetilde{g}(\nabla^{\perp}_{V}\mathbf{H}, \xi)=0$  for every $V\in\mathfrak{X}(M)$.
\item $M$ factors through a slice.
\item $\nabla^{\perp}\mathbf{H}=0$.
\end{enumerate}
\end{theorem}
\begin{proof}
Assume that $\widetilde{g}(\nabla^{\perp}_{V}\mathbf{H}, \xi)=0$. Then from Lemma  \ref {040623A} it is directly follows  that $ v$ is a constant function.  From Lemma \ref{020323A}, the function $u$ is also constant and then $M$ factors through a slice.
The mean curvature vector field of a submanifold $M$ which factors through a slice is computed from (\ref{Hxi}) as follows
$$
\mathbf{H}=\Big(\dfrac{\eta\lambda}{\lambda}\circ \Psi_{B}\Big)\xi+\Big(\dfrac{\xi\lambda}{\lambda}\circ \Psi_{B}\Big)\eta.
$$
Hence as a direct consequence of Lemma \ref{210523D} and Proposition \ref{28032023A}, we get $\nabla^{\perp}\mathbf{H}=0$.
The rest of the proof is obvious.

\end{proof}

In a similar way we have. 

\begin{theorem}\label{130623C}
Assume  $\eta (\eta \lambda)+ f f' \eta \lambda \neq 0$ at every point and let $\Psi:M\rightarrow B\times_{\lambda} F$ be a codimension two spacelike immersion  through an integral hypersurface $\mathcal{N}$  of $D_{\eta}$.  Then the following assertions are equivalent
\begin{enumerate}
\item $\widetilde{g}(\nabla^{\perp}_{V}\mathbf{H}, \eta)=0$  for every $V\in\mathfrak{X}(M)$.
\item $M$ factors through a slice.
\item $\nabla^{\perp}\mathbf{H}=0$.
\end{enumerate}
\end{theorem}

\begin{remark}
{\rm The proof of Theorems \ref{130623B} and \ref{130623C} does not work for  $\lambda(t,r)=r$.  In this case, Lemma \ref{040623A} gives 
$\widetilde{g}(\nabla^{\perp}_{V}\mathbf{H}, \xi)=0$ for a codimension two spacelike immersion through an integral hypersurface $\mathcal{L}$ of $D_{\xi}$.
Hence, the mean curvature vector field is parallel if and only if $\widetilde{g}(\nabla^{\perp}_V\mathbf{H},\eta^{\perp})=0$.
A similar result is achieved for codimension two spacelike immersions  through an integral hypersurface $\mathcal{N}$ of $D_{\eta}$.

}

\end{remark}

\noindent\small{Departamento de Matemática Aplicada, Universidad de Málaga, 29071-Málaga (Spain).}\\
\small{E-mails: rodrigometalica\_94@hotmail.com (Rodrigo Morón) and fpalomo@uma.es (Francisco J. Palomo).}\\


\begin{thebibliography}{999}


\bibitem{ACR1} L.J. Alías,  V.L. Canovas and M. Rigoli,  Trapped submanifolds contained into a null hypersurface of the de Sitter spacetime,
{\it Commum.    Contemp.  Math. ,} {\bf 20}, No. 08, (2018),  23 pp. 




\bibitem{ACR2} L.J. Alías,  V.L. Canovas and M. Rigoli,  Codimension two spacelike submanifolds 
into the light cone of Lorentz-Minkowski space,   {\it Proceedings of the Royal Society of Edinburgh Section A: Mathematics},  {\bf 149}(6),  (2018), 1523--1553.


\bibitem{AMR} L.J. Al\'ias, P. Mastrolia and M. Rigoli, 
\emph{Maximum principles and geometric applications}, Springer, 2016.



\bibitem{AlRoSan} L.J. Alías,  A. Romero and M. Sánchez,
Uniqueness of complete spacelike hypersurfaces of constant mean curvature in generalized Robertson-Walker spacetimes, {\it
Gen. Relativity Gravitation} \textbf{27} (1995), no.1, 71--84.



\bibitem{Besse}
A.L. Besse,
{\it Einstein Manifolds}, A Series of Modern Surveys in Mathematics,
Springer, 1986.


\bibitem{B} H.W. Brinkmann, Einstein spaces which are mapped conformally on each other, {\it Math. Ann.} \textbf{94} (1925), 119--145.




\bibitem{CaPaRo} V. Cánovas,  F.J. Palomo  and
A. Romero,  Mean curvature of spacelike submanifolds in a Brinkmann spacetime,
{\it Classical Quantum Gravity} \textbf{38} (2021), Paper No. 195013, 18 pp.



\bibitem{CS09} A. \v{C}ap and J. Slov\'{a}k, {\it Parabolic Geometries I}. Background and General Theory, Mathematical Surveys and Monographs \textbf{154}, AMS 2009.


\bibitem{Cha}  I. Chavel, {\it Eigenvalues in Riemannian geometry}, Academic Press, INC, Orlando,1984.



\bibitem{Chen} B.Y. Chen, {\it Geometry of Submanifolds}, Marcel Dekker,
New York, 1973.





\bibitem{DRT} D. de la Fuente, A. Romero and P.J. Torres,  Entire spherically symmetric spacelike
graphs with prescribed mean curvature
function in Schwarzschild and Reissner-Nordström spacetimes, {\it Class. Quantum Grav.}  \textbf{32} (2015), 17pp.


\bibitem{DAL} S. Deshmukh and F. Al-Solamy, Conformal gradient vector fields on a compact Riemannian manifold, {\it  Colloquium Mathematicum} \textbf{112} (2008) 157--161. 

\bibitem{Carmo1} M. P. Do Carmo. {\it  Differential geometry of curves and surfaces} (Englewood Cliffs, NJ:
Prentice Hall, 1976).



\bibitem{Gall} G. J. Galloway,  Maximum Principles for Null Hypersurfaces
and Null Splitting Theorems, {\it Ann. Henri Poincaré} (2000), 543--567.




\bibitem{Gourgoulhon} E. Gourgoulhon, {\rm Geometry and physics of black holes,} {\it Lecture notes} \url{https://relativite.obspm.fr/blackholes/bholes.pdf}. 


\bibitem{Gr} A. Grigor'yan, Analytic and geometric background of recurrence 
and non-explosion of the Brownian motion on Riemannian manifolds, 
\emph{B. Am. Math. Soc.}, \textbf{36} (1999), 135--249.

\bibitem{H} A. Huber, On subharmonic functions and differential geometry in the large, 
\emph{Comment. Math. Helv.}, \textbf{32} (1958), 13--72.

\bibitem{Ka} J.L. Kazdan, Parabolicity and the Liouville property on complete 
Riemannian manifolds, \emph{Aspects of Math.}, \textbf{10} (1987), 153--166.


\bibitem{Lee}D. D.Lee,   {\it Geometric Relativity}, {\rm Graduates Studies in Mathematics.  AMS},   \textbf{201}, Providence, 2019.

\bibitem{Leistner} T. Leistner, Screen bundles of Lorentzian manifolds and some
generalisations of pp-waves, {\it Journal of Geometry and Physics} \textbf{56} (2006), 2117--2134.




\bibitem{MS} M. Mars. and J.M.M. Senovilla, Trapped surfaces and symmetries, {\it Classical and Quantum Gravity }\textbf{20} (2003), 293-300.

\bibitem{Ob} M. Obata,  Certain conditions for a Riemannian manifold to be isometric with a sphere, {\it J. Math. Soc. Japan} \textbf{14} (1962), 333--340.


\bibitem{One66} B. O'Neill,  The fundamental equations of a submersion, {\it Michigan Math. J.}  \textbf{13} (1966), 459--469.


\bibitem{One83} B. O'Neill,  {\it Semi-Riemannian Geometry with Applications to Relativity}, {\rm Academic Press}, New York, 1983.



\bibitem{PaPaRo} O. Palmas,   F.J. Palomo  and
A. Romero,   On the total mean curvature of a compact space-like submanifold in Lorentz–Minkowski spacetime, {\it Proceedings of the Royal Society of Edinburgh,} {\bf 148A} (2018), 199--210.


\bibitem{Pal21} F.J. Palomo, Lightlike manifolds and Cartan geometries, {\it
Anal. Math. Phys.,}  \textbf{11} (2021), no.3, Paper No. 112, 39 pp.

\bibitem{PPR} F.J. Palomo, J.A. Pelegrín and A. Romero,   Rigidity results for complete spacelike submanifolds in plane fronted waves,
{\it  Rev. R. Acad. Cienc. Exactas Fís. Nat. Ser. A Mat. RACSAM} \textbf{116} (2022), Paper No. 179, 10 pp.

\bibitem{PaRoRo} F.J. Palomo,,  F. J. Rodríguez and
A. Romero,   New Characterizations of Compact Totally Umbilical Spacelike Surfaces in $4$-dimensional Lorentz–Minkowski Spacetime through a Lightcone, {\it Mediterr. J. Math. }, {\bf 11} (2014), 1229--1240.



\bibitem{PaRo} F.J. Palomo and A. Romero, On spacelike surfaces in $4$-dimensional Lorentz-Minkowski spacetime through a lightcone, {\it
Proc. Roy. Soc. Edinb. A Mat.}, {\bf 143A} (2013), 881--892.



\bibitem{Poor} W. A. Poor, {\it Differential geometric structures}, McGraw-Hill Book Co., New York, 1981.



\bibitem{WWZ} M-T. Wang, Y-K Wang and X. Zhang, Minkowski formulae and Alexandrov theorems in Spacetime, {\it J. Differential Geometry}, {\bf 105} (2017), 249--290.


\bibitem{Y} K. Yano, {\it Integral formulas in Riemannian geometry}, Pure and Applied Mathematics, No. 1 Marcel Dekker, Inc.,  New York, 1970.

\end{thebibliography}
\end{document}